\documentclass[oneside,a4paper,12pt,centertags]{amsart}
\usepackage{vmargin}
\usepackage{amssymb, amsthm, amscd, latexsym, amsmath, verbatim, mathrsfs, color}
\usepackage[all]{xy}
\usepackage[mathscr]{eucal}
\usepackage{enumitem}
\usepackage[bookmarks=true]{hyperref}

\theoremstyle{plain}
\newtheorem{Thm}{Theorem}[section]

\theoremstyle{definition}
\newtheorem{Prop}[Thm]{Proposition}
\newtheorem{Lem}[Thm]{Lemma}
\newtheorem{Cor}[Thm]{Corollary}
\newtheorem{Def}[Thm]{Definition}

\theoremstyle{remark}
\newtheorem{Rem}[Thm]{Remark}

\newcommand{\R}{\mathbb{R}}
\newcommand{\C}{\mathbb{C}}
\newcommand{\N}{\mathbb{N}}

\newcommand{\re}{\mathop{\mathrm{Re}}}
\newcommand{\im}{\mathop{\mathrm{Im}}}
\newcommand{\Arg}{\mathop{\mathrm{Arg}}}

\newcommand{\Dom}{\textsf{Dom}}
\newcommand{\Ran}{\textsf{R}}
\newcommand{\Nul}{\textsf{N}}

\newcommand{\lint}{\mathop{\lrcorner}}
\newcommand{\ext}{\mathop{\scriptstyle{\wedge}}}
\newcommand\ca{{\bf 1}}

\newcommand\cB{\mathcal{B}}

\newcommand\cD{\mathcal{D}}

\newcommand\cL{\mathcal{L}}

\newcommand\cQ{\mathcal{Q}}
\newcommand\cS{\mathcal{S}}
\newcommand\cV{\mathcal{V}}

\newcommand{\sppt}{\mathop{\mathrm{sppt}}}
\newcommand{\Div}{\mathop{\mathrm{div}}}
\newcommand{\svt}[1]{{\substack{\vspace{.025cm} \\ #1}}} 

\numberwithin{equation}{section}

\setenumerate{leftmargin=*}

\begin{document}
\title[Calder\'{o}n Reproducing Formulas ]{Calder\'{o}n Reproducing Formulas and \\Applications to Hardy Spaces}
\author[Auscher]{Pascal Auscher}
\author[McIntosh]{Alan McIntosh}
\author[Morris]{Andrew J. Morris}

\address{Pascal Auscher\\Universit\'{e} de Paris-Sud\\Laboratoire de Math\'{e}matiques\\{UMR} du {CNRS} 8628\\91405 Orsay Cedex\\France}
\email{Pascal.Auscher@math.u-psud.fr}

\address{Alan McIntosh\\Centre for Mathematics and its Applications\\Mathematical Sciences Institute\\Australian National University\\Canberra\\ACT 0200\\Australia}
\email{Alan.McIntosh@anu.edu.au}

\address{Andrew J. Morris\\Mathematical Institute\\University of Oxford\\Oxford\\OX1 3LB\\UK}
\email{Andrew.Morris@maths.ox.ac.uk}


\date{26 March 2013}

\begin{abstract}
We establish new Calder\'{o}n reproducing formulas for self-adjoint operators $D$ that generate strongly continuous groups with finite propagation speed. These formulas allow the analysing function to interact with $D$ through holomorphic functional calculus whilst the synthesising function interacts with $D$ through functional calculus based on the Fourier transform. We apply these to prove the embedding $H^p_D(\wedge T^*M) \subseteq L^p(\wedge T^*M)$, $1\leq p\leq 2$, for the Hardy spaces of differential forms introduced by Auscher, McIntosh and Russ, where $D=d+d^*$ is the Hodge--Dirac operator on a complete Riemannian manifold $M$ that has polynomial volume growth. This fills a gap in that work. The new reproducing formulas also allow us to obtain an atomic characterisation of $H^1_D(\wedge T^*M)$. The embedding $H^p_L \subseteq L^p$, $1\leq p\leq 2$, where $L$ is either a divergence form elliptic operator on~$\R^n$, or a nonnegative self-adjoint operator that satisfies Davies--Gaffney estimates on a doubling metric measure space, is also established in the case when the semigroup generated by the adjoint $-L^*$ is ultracontractive.
\end{abstract}

\maketitle

\tableofcontents

\section{Introduction and Main Results} \label{section: introduction}
The classical Hardy spaces $H^p(\R^n) \subseteq L^p(\R^n)$ provide a substitute for the $L^p(\R^n)$ scale of spaces on which homogeneous multipliers, such as the Riesz transforms $(R_ju)\, \widehat{\ }(\xi)=i{\xi_j}{|\xi|^{-1}}\widehat{u}(\xi)$ for $j\in\{1,\dots,n\}$, are bounded when $p\in[1,\infty)$. It is well known that $H^p(\R^n)=L^p(\R^n)$ when $p\in(1,\infty)$, whilst $H^1(\R^n) \subset L^1(\R^n)$, and that $H^1(\R^n)$ has an atomic characterisation and a molecular characterisation.

A variety of new Hardy spaces have been designed to obtain a similar theory for useful operators that do not belong to the standard Calder\'{o}n--Zygmund class. We are primarily motivated by the Hardy spaces of differential forms $H^p_D(\wedge T^*M)$ introduced by Auscher, McIntosh and Russ~\cite{AMcR}. We temporarily restrict our attention to these spaces, although the main content of the paper contains a more general theory that can be applied to a variety of the contexts considered elsewhere.

The $H^p_D(\wedge T^*M)$ spaces were designed for the analysis of the Hodge--Dirac operator $D=d+d^*$ and the Hodge--Laplacian $\Delta=D^2$, where $d$ and $d^*$ denote the exterior derivative and its adjoint, acting on the Hilbert space of square integrable differential forms $L^2(\wedge T^*M)$ over a complete Riemannian manifold $M$.  We will always assume that any such manifold $M$ is smooth and connected, and has doubling volume growth in the sense that there exist constants $A\geq1$ and $\kappa\geq0$ such that
\begin{equation}\tag{D$_\kappa$}\label{D1}
0<V(x,\alpha r)\leq A\alpha^{\kappa} V(x,r)<\infty\qquad \forall x\in M, \ \forall r>0, \ \forall \alpha\geq 1,
\end{equation}
where $V(x,r)$ is the Riemannian measure of the geodesic ball $B(x,r)$ in $M$ with centre $x$ and radius~$r$. These spaces were designed so that the geometric Riesz transform $D\Delta^{-1/2}$ is bounded on $H^p_D(\wedge T^*M)$ when $p\in[1,\infty]$, and a molecular characterisation was obtained for $H^1_D(\wedge T^*M)$.

One of the  aims  of this paper is to show that $H^p_D(\wedge T^*M) \subseteq L^p(\wedge T^*M)$ when $p\in[1,2]$. This result was stated in \cite[Corollary~6.3]{AMcR} but the proof contains a gap that we fill here. Another aim is to to show that $H^1_D(\wedge T^*M)$ has an atomic characterisation, thus strengthening the result in \cite[ Theorem~6.2]{AMcR}  that $H^1_D(\wedge T^*M)$ has a molecular characterisation.

We now outline the main ideas. The space $H^p_D(\wedge T^*M)$ is defined as a completion of a normed space $E^p_{D,\psi}(\wedge T^*M)$ associated with a suitably nondegenerate function $\psi$ from the set
\[
\Psi_\sigma^\tau(S_{\theta}^o) = \{\psi \in H^\infty(S_{\theta}^o\cup\{0\}): |\psi(z)| \lesssim \min\{|z|^{\sigma},|z|^{-\tau}\}\},
\]
for some $\sigma,\tau>0$, where $H^\infty(S_{\theta}^o\cup\{0\})$ denotes the algebra of bounded functions on $S_{\theta}^o\cup\{0\}$ that are holomorphic on the open bisector $S_\theta^o$ of angle $\theta\in (0,\pi/2)$ (see~\eqref{sector}). We shall not define $E^p_{D,\psi}(\wedge T^*M)$ precisely here except to mention that
\begin{equation}\label{E}
u \in E^p_{D,\psi} \quad \text{ if and only if } \quad  u= \int_0^\infty \psi_t(D) U_t \frac{dt}{t} \quad \text{ for some } U \in T^p \cap T^2,
\end{equation}
where $T^p=T^p((\wedge T^*M)_+)$ is an appropriate analogue of the tent space $T^p(\R_+^{n+1})$ introduced by Coifman, Meyer and Stein~\cite{CMS}, and $\psi_t(D)=\psi(tD)$ is defined by the holomorphic functional calculus of~$D$ (see Definition~\ref{AMcRDef5.1}).

There is an important distinction between \textit{a completion of $E^p_{D,\psi}$} and \textit{the completion of $E^p_{D,\psi}$ in $L^p$}. The former is unique up to isometric isomorphism and can always be constructed as an abstract space, whereas the latter is a unique subspace of $L^p$ that may or may not exist. See Section~\ref{section: notation} for further details. It was known previously that $E^p_{D,\psi} \subseteq L^p$ when $\psi$ has suitable decay at the origin and infinity, but this does not guarantee, nor was it proved, that the completion of $E^p_{D,\psi}$ in $L^p$ exists. Without this property, a completion of $E^p_{D,\psi}$ must be interpreted as an abstract space consisting of, for example, equivalence classes of Cauchy sequences in $E^p_{D,\psi}$ or elements of the second dual space $(E^p_{D,\psi})^{**}$. Although various realizations of such an abstract Hardy space were known, these were not shown to be contained in any function space. The approach of Hofmann, Mayboroda and McIntosh~\cite[Appendix 2]{HMMc}, for instance, can be used to realize the abstract Hardy space as a space of distributions adapted to~$D$.

We prove that the completion of $E^p_{D,\psi}(\wedge T^*M)$ in $L^p(\wedge T^*M)$ exists by utilizing the finite propagation speed of the $C_0$-group $(e^{itD})_{t\in\R}$ generated by the Hodge--Dirac operator $D$ on $L^2(\wedge T^*M)$. This provides a constant $c_D>0$ such that for all geodesic balls $B(x,r) \subseteq M$, all $u \in L^2(\wedge T^*M)$ with $\sppt(u)\subseteq B(x,r)$ and all $t \in \R$, it holds that $\sppt(e^{itD}u)\subseteq B(x,(1+c_D)r)$.

The main ideas of the argument are as follows. We use suitably nondegenerate Schwartz functions $\eta$ with compactly supported Fourier transform $\widehat{\eta}$ from the set
\[
\widetilde{\Psi}_{N}^\delta (\R) = \{\eta \in \cS(\R) : \sppt\widehat{\eta} \subseteq [-\delta,\delta] \text{ and }\partial^{k-1}\eta(0)=0 \text{ for all } k\in\{1,\ldots, N\}\},
\]
for some $\delta>0$ and $N\in\N$, to interact with the finite propagation speed of the group. We will see that for all $\eta \in \widetilde{\Psi}_{N}^\delta (\R)$ and all $u\in L^2(\wedge T^*M)$ with $\sppt(u)\subseteq B(x,r)$, it holds that $\sppt(\eta_t(D) u)\subseteq B(x,r+c_D\delta t)$, where $\eta_t(D)=\eta(tD)$ is defined by the Borel functional calculus of~$D$. This is in contrast with a function $\psi \in \Psi_\sigma^\tau(S_{\theta}^o)$, for which $\psi_t(D)u$ may be supported everywhere on $M$.

We incorporate the finite propagation speed into the existing theory by choosing $\psi \in \Psi(S^o_\theta)$ and $\eta \in \widetilde{\Psi}(\R)$ so that the following Calder\'{o}n reproducing formula holds:
\begin{equation}\label{rep}
\int_0^\infty \psi_t(D) \eta_t(D)u \frac{dt}{t} = \int_0^\infty \eta_t(D) \psi_t(D)u \frac{dt}{t} = u \qquad \forall u \in E^p_{D,\psi} \cup E^p_{D,\eta} .
\end{equation}
A comparison of \eqref{E} and \eqref{rep} shows that if $u\in E^p_{D,\eta}$ and ${\eta_t(D)u \in T^p \cap T^2}$, then $u\in E^p_{D,\psi}$. This principle allows us to prove that $E^p_{D,\psi} = E^p_{D,\eta}$ when the family of operators $(\psi_t(D) \eta_s(D))_{s,t\in(0,\infty)}$ has enough $L^2$ off-diagonal decay to control volume growth on the manifold. We then use the Sobolev embedding theorem for compact manifolds and standard energy estimates for the group $(e^{itD})_{t\in\R}$ to prove that the completion of $E^p_{D,\eta}$ in $L^p$ exists, hence the completion of $E^p_{D,\psi}$ in $L^p$ exists as well.

Let us remark that the connection between the classical Hardy spaces $H^p(\R^n)$ and the tent spaces $T^p(\R^{n+1}_+)$ was previously understood in terms of reproducing formulas analogous to~\eqref{rep} for convolution operators. In particular, Coifman, Meyer and Stein provided a short proof of the atomic characterisation of $H^p(\R^n)$ for $p\in(0,1]$ in \cite[Section 9b]{CMS} by using the theory of tent spaces and constructing a function $\phi \in C^\infty_c(\R^n)$ satisfying $\int x^\gamma\phi(x)\ dx = 0$ for all $\gamma\in[0,N_p]$ and some $N_p\in\N$ depending on $p$ such that
\[
\int_0^\infty \phi_{(t)} * \partial_t P_{(t)}*f  \ dt = f \qquad \forall f\in H^p(\R^n),
\]
where $P$ is the Poisson kernel and $P_{(t)}(x)=t^{-n}P(x/t)$. This is equivalent to 
\[
\int_0^\infty \widehat\phi(t\xi)(-2\pi t|\xi|)e^{-2\pi t|\xi|}\,\frac{dt}{t} = 1 \qquad \forall \xi\in\R^n \setminus\{0\},
\]
from which the analogy with \eqref{rep} is most apparent when $n=1$, since $\eta(x):=\widehat{\phi}(x)$ is in $\widetilde{\Psi}_{N_p+1}^\delta(\R)$ for some $\delta>0$, whilst $\psi(z):=\begin{cases}-2\pi z e^{-2\pi z}, & \text{if } \re(z)\geq 0 \\ 2\pi z e^{2\pi z}, & \text{if } \re(z)<0\end{cases}$ is in $\Psi_1^\tau(S^o_\theta)$ for all $\tau>0$ and $\theta\in(0,\pi/2)$.

After we establish the embedding $H^1_D(\wedge T^*M)\subseteq L^1(\wedge T^*M)$, the finite propagation speed of the group $(e^{itD})_{t\in\R}$ also allows us to obtain an atomic characterisation of~$H^1_D(\wedge T^*M)$. This builds on the molecular characterisation obtained in~\cite{AMcR}. The molecular space $H^1_{D,\text{mol}(N)}(\wedge T^*M)$ and the atomic space $H^1_{D,\text{at}(N)}(\wedge T^*M)$ are introduced in Definition~\ref{AMcRDef6.1}, where $N\in\N$ is the number of moment conditions satisfied by the molecules and atoms in the respective spaces.

The following theorem summarizes our results for the Hodge--Dirac operator.

\begin{Thm}\label{Thm: HpinLpIntro}
Suppose that $M$ is a complete Riemannian manifold satisfying~\eqref{D} and that $D=d+d^*$ is the Hodge--Dirac operator on $L^2(\wedge T^*M)$. If $p\in[1,2]$, $\theta\in(0,\pi/2)$, $\beta>\kappa/2$ and $\psi\in\Psi_{\beta}(S^o_\theta)$ is nondegenerate, then the completion $H^p_{D,\psi}(\wedge T^*M)$ of $E^p_{D,\psi}(\wedge T^*M)$ in $L^p(\wedge T^*M)$ exists. Moreover, if $N\!\in\!\N$ and $N\!>\!\kappa/2$, then $H^1_{D,\psi}(\wedge T^*M)=H^1_{D,\text{mol}(N)}(\wedge T^*M)=H^1_{D,\text{at}(N)}(\wedge T^*M)$.
\end{Thm}

The Hardy space $H^p_{D,\psi}(\wedge T^*M)$ in Theorem~\ref{Thm: HpinLpIntro} is thus the set of all $u$ in $L^p(\wedge T^*M)$ for which there exists a Cauchy sequence $(u_n)_n$ in $E^p_{D,\psi}(\wedge T^*M)$ that converges to $u$ in $L^p(\wedge T^*M)$, together with the norm $\|u\|_{H^p_{D,\psi}} = \lim_n\|u_n\|_{E^p_{D,\psi}}$. The embedding $H^p_{D,\psi}(\wedge T^*M) \subseteq L^p(\wedge T^*M)$ is then automatic. The comments below Definition~\ref{Def: relcompletion} contain more details. 

The results obtained here can also be applied to Hardy spaces designed for higher order operators. In particular, consider the Hardy spaces $H^p_{L,\psi}(\R^n)$ introduced by Hofmann, Mayboroda and McIntosh~\cite{HMMc} for the analysis of divergence form operators $L = -\Div A \nabla = -\sum_{j,k=1}^n\partial_j A_{jk}\partial_k$, acting on $L^2(\R^n)$ and interpreted in the usual weak sense via a sesquilinear form, where $A=(A_{jk}) \in L^\infty(\R^n,\cL(\C^n))$ is elliptic in the sense that there exists $\lambda>0$ such that 
\begin{equation}\label{ellip}
\re \langle A(x)\zeta,\zeta\rangle_{\C^n}\geq\lambda|\zeta|^2 \qquad \forall \zeta\in\C^n, \ \text{ a.e. } x\in\R^n.
\end{equation}
There exists $\omega_L\in[0,\pi/2)$ such that $L$ is $\omega_L$-sectorial, hence $-L$ and $-L^*$ generate analytic semigroups $(e^{-tL})_{t>0}$ and $(e^{-tL^*})_{t>0}$ on $L^2(\R^n)$. In order to embed $H^p_{L,\psi}(\R^n)$ in $L^p(\R^n)$ when $1\leq p\leq 2$, we assume that there exists $g\in L^{ 2}_{\text{loc}}((0,\infty))$ such that
\begin{equation}\label{Gnew}
\|e^{-tL^*}u\|_\infty \leq g(t) \|u\|_2 \qquad \forall u\in L^2(\R^n).
\end{equation}

Let us remark that \eqref{Gnew} is equivalent to the action of the semigroup $(e^{-tL})_{t>0}$ from $L^1(\R^n)$ to $L^2(\R^n)$ (it is usually called ultracontractivity). Hence, this action of the semigroup on $L^1(\R^n)$ suffices to obtain $H^1_{L,\psi}(\R^n)$ as a subspace of $L^1(\R^n)$ in Theorem~\ref{Thm: HpinLpIntroL} below.

Let us also remark that \eqref{Gnew} is immediate when the semigroup $(e^{-tL^*})_{t>0}$ has a kernel $(K_t(\cdot,\cdot))_{t>0}$ defined pointwise almost everywhere on $\R^n\times\R^n$ with the property that for each $T>0$, there exist constants $C_T,c_T>0$ such that
\begin{equation}\label{Gloc}
|K_t(x,y)| \leq C_T t^{-n/2} e^{-c_T|x-y|^2/t} \qquad \forall x,y\in\R^n,\ \forall t\in(0,T].
\end{equation}
In fact, property \eqref{Gnew} is usually obtained as a step toward proving \eqref{Gloc}. For example, the local Gaussian estimates in~\eqref{Gloc} hold when, in addition to  having  $A$ bounded and elliptic, $A$  is uniformly continuous  (see \cite[Theorem~4.8]{AuscherJLMS1996}) or belongs to ~\textit{VMO} or has small \textit{BMO}~norm (see \cite[Chapter~1]{AT}).

The following theorem is essentially known when \eqref{Gloc} holds (see the remark below Proposition~9.1 in \cite{HMMc}). We provide a short proof when \eqref{Gnew} holds as an application of our techniques. 

\begin{Thm}\label{Thm: HpinLpIntroL}
Suppose that $A \in L^\infty(\R^n,\cL(\C^n))$ is elliptic and that $L=-\Div A \nabla$ on $L^2(\R^n)$ satisfies  \eqref{Gnew}. If $p\in[1,2]$, $\theta\in(\omega_L,\pi/2)$, $\beta>n/4$ and $\psi\in\Psi_\beta(S^o_\theta)$ is nondegenerate, then the completion $H^p_{L,\psi}(\R^n)$ of $E^p_{L,\psi}(\R^n)$ in $L^p(\R^n)$ exists. Moreover, if $N\in\N$ and $N>n/4$, then $H^1_{L,\psi}(\R^n)=H^1_{L,\text{mol}(N)}(\R^n)$,  and when $A$ is self-adjoint, then also $H^1_{L,\psi}(\R^n)=H^1_{L,\text{at}(N)}(\R^n)$.
\end{Thm}

A theory of Hardy spaces was developed by Hofmann, Lu, Mitrea, Mitrea and Yan~\cite{HLMMY} for nonnegative self-adjoint operators $L$ satisfying Davies--Gaffney estimates (see \eqref{DG}) on doubling metric measure spaces $M$. For example, when $A$ is self-adjoint, then $L=-\Div A \nabla$ has these properties. The framework developed here provides an embedding for these spaces when $L$ acts on a vector bundle $\cV$ over $M$, as defined in Section~\ref{section: notation}, and there exists $g\in L^{ 2}_{\text{loc}}((0,\infty))$ such that
\begin{equation}\label{GnewSA}
\|e^{-tL}u\|_\infty \leq g(t) \|u\|_2 \qquad \forall u\in L^2(\cV).
\end{equation}
In this context, since $L$ is self-adjoint, it is well known that \eqref{GnewSA} is equivalent to pointwise kernel estimates for the semigroup $(e^{-tL})_{t>0}$ (see \cite[Lemma~2.1.2]{Davies1989}).

\begin{Thm}\label{Thm: HpinLPforHLMMY}
Suppose that $M$ is a doubling metric measure space satisfying~\eqref{D} and that $L$ is a nonnegative self-adjoint operator on $L^2(\cV)$ satisfying Davies--Gaffney estimates and \eqref{GnewSA}. If $p\in[1,2]$, $\theta\in(0,\pi/2)$, $\beta>\kappa/4$ and $\psi\in\Psi_\beta(S^o_\theta)$ is nondegenerate, then the completion $H^p_{L,\psi}(\cV)$ of $E^p_{L,\psi}(\cV)$ in $L^p(\cV)$ exists. Moreover, if $N\in\N$ and $N>\kappa/4$, then $H^1_{L,\psi}(\cV)=H^1_{L,\text{mol}(N)}(\cV)=H^1_{L,\text{at}(N)}(\cV)$.
\end{Thm}

It remains an open question as to whether Theorems~\ref{Thm: HpinLpIntroL} and~\ref{Thm: HpinLPforHLMMY} hold in the absence of ultracontractivity estimates such as~\eqref{Gnew} and \eqref{GnewSA}. The first-order methods developed here, however, provide a new proof of Theorem~\ref{Thm: HpinLpIntroL} that does not rely on ultracontractivity but instead requires that $A$ is self-adjoint with smooth coefficients. We present this proof at the conclusion of the paper as a basis for future work.

The structure of the paper is as follows. In Section~\ref{section: notation}, we fix notation and discuss when the completion of a normed space inside a given Banach space exists. In Section~\ref{section: Pointwise}, we briefly recast the theory of Hardy spaces from~\cite{AMcR} in the context of a vector bundle $\cV$ over a doubling metric measure space $M$ for any operator ${\cD}$ on $L^2(\cV)$ that is bisectorial with a bounded holomorphic functional calculus and that satisfies polynomial off-diagonal estimates. We then introduce an additional hypothesis \ref{SFinLq} on ${\cD}$, based on the $\Psi(S^o_\theta)$ class, that guarantees the embedding $H^p_{\cD}(\cV) \subseteq L^p(\cV)$, when $p\in[1,2]$, and the molecular characterisation of $H^1_{\cD}(\cV)$. This is the content of Theorems~\ref{Thm: H1injectivity} and \ref{AMcRThm6.2}.

In Section~\ref{section: FPS}, we restrict consideration to any operator $D$ that is self-adjoint on $L^2(\cV)$ and for which the associated $C_0$-group $(e^{itD})_{t\in\R}$ has finite propagation speed. This allows us to introduce an alternative hypothesis \ref{SFinLqFPS} on $D$, based on the $\widetilde\Psi(\R)$ class, that guarantees the embedding $H^p_D(\cV) \subseteq L^p(\cV)$, when $p\in[1,2]$, and the atomic characterisation of $H^1_D(\cV)$. This is the content of Theorems~\ref{Thm: H1injectivityFPS} and~\ref{AMcRThm6.2FPS}. In Theorem~\ref{Thm: HpinLpD}, we verify \ref{SFinLqFPS} when $M$ is a complete Riemannian manifold and $D$ is a smooth-coefficient, self-adjoint, first-order, differential operator with bounded principal symbol

The results for the Hodge--Dirac operator $D=d+d^*$ and the divergence form operator $L=-\Div A \nabla$ in Theorems~\ref{Thm: HpinLpIntro} and~\ref{Thm: HpinLpIntroL} are deduced in Sections~\ref{section: ApplicationsI} and~\ref{section: ApplicationsII}. In Section~\ref{section: HLMMY}, we combine the techniques of the preceding two sections to prove Theorem~\ref{Thm: HpinLPforHLMMY}. Section~\ref{section:ODE} is an appendix that contains the technical off-diagonal estimates used to prove Theorems~\ref{Thm: H1injectivityFPS} and \ref{AMcRThm6.2FPS}.

\section{Notation and Preliminaries}\label{section: notation}
Throughout the paper, let $M$ denote a {\it metric measure space} with a metric $\rho$ and a $\sigma$-finite measure $\mu$ that is Borel with respect to the $\rho$-topology. A ball in $M$ will always refer to an open $\rho$-ball. For $x\in M$ and $\alpha, r>0$, let $B(x,r)$ denote the ball in $M$ with centre $x$ and radius $r$, let $V(x,r) = \mu(B(x,r))$ and $(\alpha B)(x,r)=B(x,\alpha r)$. The metric measure space $M$ is called \textit{doubling} when there exist constants $A\geq1$ and $\kappa\geq0$ such that \begin{equation}\tag{D$_\kappa$}\label{D}
0<V(x,\alpha r)\leq A\alpha^{\kappa} V(x,r)<\infty\qquad \forall x\in M, \ \forall r>0, \ \forall \alpha\geq 1.
\end{equation}
For any $E, F\subseteq M$, set $\rho(E,F) = \inf \{ \rho(x,y) : x\in E,\, y\in F\}$.

A {\it vector bundle} $\cV$ over $M$ refers to a  complex vector bundle $\pi:\cV\rightarrow M$ equipped with a Hermitian metric $\langle\cdot,\cdot\rangle_x$ that depends continuously on $x\in M$. For any vector bundle $\cV$, there are naturally defined Banach spaces $L^p(\cV)$, $1\leq p\leq \infty$, of measurable sections. The Hilbert space $L^2(\mathcal{V})$ of square integrable sections of $\mathcal{V}$ has the inner product $\langle u, v \rangle = \int_M \langle u(x), v(x) \rangle_x\, d\mu(x)$. For any linear operator $T$ on $L^2(\cV)$, the domain $\Dom(T)$, range $\Ran(T)$ and null space $\Nul(T)$ are subspaces of $L^2(\cV)$, and the operator norm $\|T\| = \sup \{\|Tu\|_{L^2(\mathcal{V})}/\|u\|_{L^2(\mathcal{V})} : u\in\Dom(T), u\neq0 \}$. The Banach algebra of all bounded linear operators on $L^2(\mathcal{V})$ is denoted by $\mathcal{L}(L^2(\mathcal{V}))$.

For normed spaces $X$ and $Y$, we write $X\subseteq Y$ when $X$ is a subset of $Y$ with the property that there exists $C>0$ such that $\|x\|_Y\leq C \|x\|_X$ for all $x\in X$, and we write $X=Y$ when $X\subseteq Y \subseteq X$. A completion $(\mathcal{X},\imath)$ of a normed space $X$ consists of a Banach space $\mathcal{X}$ and an isometry $\imath:X\rightarrow\mathcal{X}$ such that $\imath(X)$ is dense in $\mathcal{X}$. Every normed space has a completion but this abstract construction is not sufficient for our purposes. It is convenient to formalise the following related notion.

\begin{Def}\label{Def: relcompletion}
Let $X$ be a normed space and suppose that $X\subseteq Y$ for some Banach space $Y$. A Banach space $\widetilde{X}$ is called \textit{the completion of $X$ in $Y$} when $X\subseteq\widetilde{X}\subseteq Y$, the set $X$ is dense in~$\widetilde{X}$, and  $\|x\|_X = \|x\|_{\widetilde{X}}$ for all $x\in X$. 
\end{Def}  

It is easily checked that the completion $\widetilde{X}$ of $X$ in $Y$ is unique whenever it exists. Moreover, the set $\widetilde{X}$ consists of all $x$ in $Y$ for which there is a Cauchy sequence $(x_n)_n$ in $X$ such that $(x_n)_n$ converges to $x$ in $Y$, and the norm $\|x\|_{\widetilde{X}} = \lim_{n\rightarrow\infty} \|x_n\|_{X}$. This can be deduced from the following necessary and sufficient conditions for the existence of a completion inside a given Banach space. The proof is left to the reader.

\begin{Prop}\label{Prop: relativecomp}
Let $X$ be a normed space and suppose  that $X\subseteq Y$ for some Banach space $Y$, so the identity $I:X\rightarrow Y$ is bounded. The following are equivalent:
\begin{enumerate}
\item the completion of $X$ in $Y$ exists;
\item if $(\mathcal{X},\imath)$ is a completion of $X$, then the unique operator $\widetilde{I}$ in $\mathcal{L}(\mathcal{X},Y)$ defined by the commutative diagram below, is injective;
\[\label{fig1}
\xymatrix{
X \ar[d]_-{{\imath}} \ar[r]^-{I}
& Y\\
\mathcal{X}  \ar[ur]_-{\widetilde{I}}
}\]
\item for each Cauchy sequence $(x_n)_n$ in $X$ that converges to $0$ in $Y$, it follows that $(x_n)_n$ converges to $0$ in $X$.
\end{enumerate}
\end{Prop}

We adopt the convention for estimating $x,y\geq0$ whereby $x\lesssim y$ means that there exists a constant $C\geq1$, which only depends on constants specified in the relevant preceding hypotheses, such that $x\leq Cy$. We write $x\eqsim y$ when $x\lesssim y \lesssim x$. The set of positive integers is denoted by $\mathbb{N}$ whilst $\mathbb{N}_0=\mathbb{N}\cup\{0\}$ and $\R_+=(0,\infty)$. Finally, we apologise in advance for the excess of notation, but it is required to handle some delicate points.

\section{Sectorial Operators with Off-Diagonal Estimates}\label{section: Pointwise}
Auscher, McIntosh and Russ~\cite{AMcR} designed the Hardy spaces of differential forms $H^p_D(\wedge T^*M)$, $1\leq p\leq \infty$, for the Hodge--Dirac operator $D=d+d^*$ acting on $L^2(\wedge T^*M)$ over a doubling Riemannian manifold $M$. We briefly recast that theory in the context of a vector bundle $\cV$ over a doubling metric measure space $(M,\rho,\mu)$. Instead of the Hodge--Dirac operator, we consider any closed, densely defined operator $\cD: \Dom(\cD) \subseteq L^2(\cV) \rightarrow L^2(\cV)$ that is bisectorial with a bounded holomorphic functional calculus (e.g. this holds when $\cD$ is self-adjoint) and satisfies polynomial off-diagonal estimates (e.g. these hold for suitable classes of differential operators~$\cD$, not necessarily of first-order). The setup below allows us to define these properties.

For $0\leq\mu<\theta<\pi/2$, define the following bisectors in the complex plane:
\begin{align}\begin{split}\label{sector}  
S_\mu &= \{z\in\C : z=0 \text{ or } |\arg z|\leq\mu \text{ or } |\pi-\arg z|\leq\mu\};\\
S_\theta^o &= \{z\in\C\setminus\{0\} : |\arg z|<\theta \text{ or } |\pi-\arg z|<\theta\}.
\end{split}\end{align}
A function on $S_\theta^o$ is called \textit{nondegenerate} when it is not identically zero on each component of $S_\theta^o$. The algebra of bounded complex-valued functions on $S_{\theta}^o\cup\{0\}$ that are holomorphic on $S_\theta^o$ is denoted by $H^\infty(S_\theta^o\cup\{0\})$. For $\sigma,\tau>0$, define
\begin{align*}
\Psi_\sigma^\tau(S_{\theta}^o) = \{\psi \in H^\infty(S_{\theta}^o\cup\{0\}): |\psi(z)| \lesssim \min\{|z|^{\sigma},|z|^{-\tau}\}\},
\end{align*}
$\Psi_\sigma(S_{\theta}^o)=\bigcup_{\tau>0}\Psi_\sigma^\tau(S_{\theta}^o)$, $\Psi^\tau(S_{\theta}^o)=\bigcup_{\sigma>0}\Psi_\sigma^\tau(S_{\theta}^o)$ and $\Psi(S_{\theta}^o)=\bigcup_{\sigma>0}\bigcup_{\tau>0}\Psi_\sigma^\tau(S_{\theta}^o)$. For functions $f:S_{\theta}^o\rightarrow\C$, define $f^*(z)=\overline{f(\bar z)}$, and for $t>0$, define $f_t(z)=f(tz)$.

Consider the following hypotheses concerning a closed, densely defined operator $\cD : \Dom(\cD) \subseteq L^2(\cV) \rightarrow L^2(\cV)$, where $\ca_E$ denotes the characteristic function of a measurable set $E\subseteq M$, and $\langle \alpha \rangle = \min \{\alpha, 1\}$ and $\langle \frac{\alpha}{0} \rangle = 1$ when $\alpha>0$.
\begin{equation}\tag{H1}\label{H1}
\begin{minipage}[t]{0.915\textwidth}
There exists $\omega\in[0,\pi/2)$ such that $\cD$ is $\textit{type }S_{\omega}$, which is defined to mean that the spectrum $\sigma(\cD) \subseteq S_{\omega}$ and that for each $\theta\in(\omega,\pi/2)$, there exists $C_\theta>0$ such that  $\|(zI-\cD)^{-1}u\|_2 \leq {C_\theta}\|u\|_2/{|z|}$ for all $z\in \C\setminus S_\theta$ and $u\in L^2(\cV)$.
\end{minipage}
\end{equation}
\begin{equation}\tag{H2}\label{H2}
\begin{minipage}[t]{0.915\textwidth}
For each $\theta\in(\omega,\pi/2)$, the operator $\cD$ has a \textit{bounded $H^\infty(S_{\theta}^o\cup\{0\})$ functional calculus in $L^2(\cV)$}, which is defined to mean that there exists $c_\theta>0$ such that  $\|\psi(\cD) u\|_2\leq c_\theta \|\psi\|_\infty\|u\|_2$ for all $\psi\in \Psi(S^o_\theta)$ and ${u\in L^2(\cV)}$.
\end{minipage}
\end{equation}
\begin{equation}\tag{H3}\label{H3}
\begin{minipage}[t]{0.915\textwidth}
There exists $m\in\N$ such that for each $\theta\in(\omega,\pi/2)$ and $N\in\N$ it holds that 
\[
 \|\ca_E (zI-\cD)^{-1} \ca_F u \|_2 \leq  \frac{C_{\theta,N}}{|z|} \left\langle \frac{1}{\rho(E,F)^{m}|z|} \right\rangle^N  \|u\|_2
\]
for all $z\in\C\setminus S_\theta$, $u\in L^2(\cV)$, measurable sets $E,F\subseteq M$, and some $C_{\theta,N}>0$.
\end{minipage}
\end{equation}

Let us note that \eqref{H1} is implicit in \eqref{H2} and \eqref{H3}. It is well known that \eqref{H1} and \eqref{H2} hold with $\omega=0$, $C_\theta=1/\sin\theta$ and $c_\theta=1$, whenever $\cD$ is self-adjoint. The number $m$ in \eqref{H3} indicates that the off-diagonal estimates associated with $\cD$ resemble those associated with an $m$th-order differential operator.

The theory of type $S_{\omega}$ operators is well known (see, for instance, \cite{McIntosh1986,ADMc,AMcN}). If \eqref{H1} holds, then for $\theta \in (\omega,\pi/2)$ and $\psi\in\Psi(S^o_\theta)$, define $\psi(\cD)\in\mathcal{L}(L^2(\cV))$ by
\begin{equation}\label{eq:Pfc}
\psi(\cD)u = \frac{1}{2\pi i} \int_{\partial S^o_{\mu}} \psi(z) (zI-\cD)^{-1}u\, d z \qquad \forall u\in L^2(\cV),
\end{equation}
where $\mu\in(\omega,\theta)$ is arbitrary and $\partial S^o_{\mu}$ is the positively oriented boundary of $S^o_{\mu}$. It holds that $L^2(\cV) = \overline{\Ran(\cD)} \oplus N(\cD)$ when $\cD$ is type $S_{\omega}$ (see \cite[Theorem~3.8]{CDMcY}) and so
\begin{equation}\label{eq:PfcRan}
\psi(\cD)u = \mathsf{P}_{\overline{\Ran(\cD)}}\, \psi(\cD)\, \mathsf{P}_{\overline{\Ran(\cD)}}\, u \qquad \forall u\in L^2(\cV),
\end{equation}
where $\mathsf{P}_{\overline{\Ran(\cD)}}$ denotes the projection from $L^2(\cV)$ onto $\overline{\Ran(\cD)}$ (see~\cite[Lemma~4.5]{Morris2010}).

It is well known (see \cite{ADMc,McIntosh1986}) that \eqref{H2} holds if and only if the quadratic estimate
\begin{equation}\label{eq:QE}
\int_0^\infty \|\psi_t(\cD)u\|_2^2\ \frac{ d t}{t} \eqsim \|u\|^2 \qquad \forall u\in\overline{\Ran(\cD)}
\end{equation}
holds for all nondegenerate $\psi\in\Psi(S^o_{\theta})$, where $\psi_t(z)=\psi(tz)$. If \eqref{H2} holds, then for $f\in H^\infty(S^o_{\theta}\cup\{0\})$, define $f(\cD)\in\mathcal{L}(L^2(\cV))$ satisfying $\|f(\cD)\|\leq c_\theta \|f\|_\infty$ by
\begin{equation}\label{eq:Hfc}
f(\cD)u = \lim_{n\rightarrow\infty} (f\psi_{(n)})(\cD)u + f(0)\mathsf{P}_{N(\cD)}u\qquad \forall u\in L^2(\cV),
\end{equation}
where $(\psi_{(n)})_{n\in\N}$ is an arbitrary sequence of uniformly bounded functions in $\Psi(S^o_{\theta})$ that converges to~1 uniformly on compact subsets of $S^o_{\theta}$. The mapping $f~\mapsto~f(\cD)$ given by~\eqref{eq:Hfc} is the unique algebra homomorphism from $H^\infty(S_{\theta}^o\cup\{0\})$ into $\mathcal{L}(L^2(\cV))$ with the following properties (see~\cite[Lecture 2]{ADMc}):  
\begin{align}
&\label{IdH}\textrm{if $\mathbf{1}(z)=1$ on $S^o_\theta\cup\{0\}$, then $\mathbf{1}(\cD)=I$ on $L^2(\cV)$;}\\
&\label{ResH}\textrm{if $\lambda\in \C\setminus S_\omega$ and $f(z)=(\lambda-z)^{-1}$ on $S^o_\theta \cup \{0\}$, then $f(\cD)=(\lambda I - \cD)^{-1}$;}\\
&\label{cvH}\textrm{\begin{minipage}[t]{0.9\textwidth}if $(f_n)_n$ is a sequence in $H^\infty(S_\theta^o\cup\{0\})$ that converges uniformly on compact sets to a function $f$ in $H^\infty(S_\theta^o\cup\{0\})$, and $\textstyle \sup_n \|f_n\|_{\infty} < \infty$, then $\textstyle \lim_n f_n(\cD)u = f(\cD)u$ for all $u\in L^2(\cV)$.
\end{minipage}}
\end{align}

Hypotheses \eqref{H1}--\eqref{H3} are sufficient to construct Hardy spaces $H^p_\cD(\cV)$ as in~\cite{AMcR}. To begin, we use~\eqref{eq:Pfc} to obtain the following extension of \cite[Lemma~3.6]{AMcR} (for the improved $\Psi(S^o_\theta)$ class exponents presented here, see~\cite[Lemma~ 7.3]{HvNP}): if $0<\delta<\sigma$, $\theta\in(\omega,\pi/2)$ and $\psi\in\Psi_\sigma(S^o_\theta)$, then there exists $C>0$ such that 
\begin{equation}\label{eq:PfcOD}
 \|\ca_E (f\psi_t)(\cD) \ca_F u\|_2 \leq C\|f\|_\infty \left\langle \frac{t}{\rho(E,F)^{m}} \right\rangle^{\sigma-\delta}\|u\|_2
\end{equation}
for all $t>0$, $f\in H^\infty(S^o_\theta\cup\{0\})$, $u\in L^2(\cV)$, and measurable sets $E,F\subseteq M$.

The theory of tent spaces $T^p(\R^{n+1}_+)$ developed by Coifman, Meyer and Stein~\cite{CMS} has the following extension when $\pi:\mathcal{V}\rightarrow M$ is a vector bundle over a \textit{doubling} metric measure space $M$. Let $\mathcal{V_+}$ denote the vector bundle $\pi_+:\mathcal{V}\times\R_+\rightarrow M\times\R_+$ over $M\times\R_+$ defined by $\pi_+(v,t):=(\pi(v),t)$ for all $v\in\cV$, $t\in\R_+$. For $x\in M$, $t\in\R_+$ and sections $U, V$ of $\mathcal{V_+}$, since $U(x,t)\in \pi^{-1}(\{x\}) \times \{t\}$, we let $U_t(x)$ denote the component of $U(x,t)$ in $\pi^{-1}(\{x\})$, and define the Hermitian metric on $\cV_+$ by $\langle U(x,t),V(x,t) \rangle_{x,t} := \langle U_t(x), V_t(x)\rangle_x$. For $p\in[1,\infty)$, the tent space $T^p(\mathcal{V}_+)$ is the Banach space of all $U$ in $L^2_{\text{loc}}(\mathcal{V}_+)$ satisfying
\[
\|U\|_{T^p} := \left(\int_M \bigg(\iint_{\Gamma(x)} |U_t(y)|_{y}^2\ \frac{ d\mu(y)}{V(y,t)} \frac{ d t}{t}\bigg)^{{p}/{2}}d\mu(x) \right)^{{1}/{p}} < \infty,
\]
where the cone $\Gamma(x) = \{(y,t)\in M\times \R_+\ |\ \rho(x,y)<t\}$. The tent space ${T^\infty(\mathcal{V}_+)}$ is the Banach space of all $U$ in $L^2_{\text{loc}}(\mathcal{V}_+)$ satisfying
\[
\|U\|_{T^\infty} := \sup_{x\in M} \sup_{B\in\mathcal{B}(x)}\bigg(\frac{1}{\mu(B)}\iint_{T(B)} |U_t(y)|_{y}^2\ d\mu(y)\frac{ d t}{t}\bigg)^{{1}/{2}} < \infty,
\]
where $\mathcal{B}(x)$ denotes the set of all balls $B\subseteq M$ with the property that $x\in B$, and the tent $T(B) = \{(y,t)\in M\times\R_+\ |\ \rho(y,M\setminus B)\geq t\}$.

We require the following properties, which can be proved as in the references cited when $M$ is a \textit{doubling} metric measure space: 
\begin{align}
\label{TpDual}\textrm{\begin{minipage}[t]{0.9\textwidth}
if $p\in[1,\infty)$ and $1/p+1/p'=1$, then $T^{p'}$ is realized as the dual of $T^p$ by the pairing $\textstyle\langle U, V \rangle_{T^2} := \int_0^\infty\int_M \langle U_t(x),V_t(x) \rangle_x\, d\mu(x)dt/t$ (see \cite[Theorem 1]{CMS});
\end{minipage}}\\
\label{TpInterp}\textrm{\begin{minipage}[t]{0.9\textwidth}
if $\theta\in(0,1)$, $1\leq p_0 < p_1\leq \infty$ and $1/p_\theta=(1-\theta)/p_0+\theta/p_1$, then the complex interpolation space $[T^{p_0},T^{p_1}]_\theta = T^{p_\theta}$ (see \cite{HVT, Bernal, CV, Amenta}).
\end{minipage}}
\end{align}

There is also the following atomic characterisation of  $T^1(\mathcal{V}_+)$, for which a section $A \in L^2(\mathcal{V}_+)$ is called a $T^1$-\textit{atom} when there is a ball $B \subseteq M$ such that $A$ is supported on the tent $T(B)$ and the norm $\|A\|_{T^2}\leq \mu(B)^{-1/2}$.

\begin{Thm}\label{maintentatomic}
Suppose that $\cV$ is a vector bundle over a doubling metric measure space $M$ and that $p\in[1,\infty)$. For each $U$ in $T^1(\mathcal{V}_+)\cap T^p(\mathcal{V}_+)$, there exist a sequence $(\lambda_j)_j$ in $\ell^1$ and a sequence $(A_j)_j$ of $T^1$-atoms such that $\sum_j\lambda_jA_j$ converges to $U$ in $T^1(\mathcal{V}_+)$, in $T^p(\mathcal{V}_+)$ and almost everywhere in $M\times\R_+$, such that $\|U\|_{T^{1}} \eqsim \|(\lambda_j)_j\|_{\ell^1}$.
\end{Thm}

\begin{proof}
This follows the proof in \cite[Theorem~1.1]{Russ}, which is based on~\cite[Theorem~1]{CMS}. The convergence in $T^p$ is not explicit in those references, but it follows by dominated convergence, as in~\cite[Proposition~3.25]{HMMc} or \cite[Theorem~3.6]{CMcM}.
\end{proof}

We follow~\cite{AMcR} to begin the development of Hardy spaces $H^p_\cD(\cV)$ in earnest.

\begin{Def}\label{Def: QandS}
Suppose that $\cD$ satisfies \eqref{H1}--\eqref{H3} on $L^2(\cV)$ for some $\omega\in[0,\pi/2)$ and $m \in\N$. For $\theta \in (\omega,\pi/2)$ and $\psi\in\Psi(S_{\theta}^o)$, define $\cQ _\psi^\cD$ in $\mathcal{L}(L^2,T^2)$ by
\[
(\cQ _\psi^\cD u)_t = \psi(t^{m}\cD)u \qquad \forall t>0,\ \forall u\in L^2(\mathcal{V})
\]
and $\cS _\psi^\cD$ in $\mathcal{L}(T^2,L^2)$ by
\[
\cS _{\psi}^\cD U = \int^\infty_0\psi(s^{m}\cD) U_s\frac{ d s}{s} \qquad \forall U\in T^2(\mathcal{V}_+).
\]
\end{Def}

The operator $\cQ_\psi^\cD$ is bounded because \eqref{H2} is equivalent to the quadratic estimate in~\eqref{eq:QE}. The operator $\cS_\psi^\cD$ is bounded because $\cS_\psi^\cD= (\cQ_{\psi^*}^{\cD^*})^*$ and the adjoint $\cD^*$ satisfies \eqref{H2} if and only if $\cD$ satisfies \eqref{H2} (see, for instance, \cite[Lecture 3]{ADMc}). These operators provide the following Calder\'{o}n reproducing formula (see~\cite[Remark 2.1]{AMcR}).

\begin{Prop}\label{Prop: AMcR_Remark2.1}
Suppose that $\cD$ satisfies \eqref{H1}--\eqref{H3} for some $\omega\in[0,\pi/2)$ and $m \in\N$. If $\sigma,\tau>0$, $\theta\in(\omega,\pi/2)$ and $\psi\in\Psi(S_\theta^o)$ is nondegenerate, then there exists a nondegenerate $\tilde\psi\in\Psi_\sigma^\tau(S_\theta^o)$ such that $\cS_{\svt{\psi}}^\cD\cQ_{\tilde\psi}^\cD u = \cS_{\tilde\psi}^\cD\cQ_{\svt{\psi}}^\cD u = \mathsf{P}_{\overline{\Ran(\cD)}}\,u$ for all $u\in L^2(\cV)$.
\end{Prop}

In preparation for defining the Hardy space $H^p_{\cD,\psi}(\cV)$, we now define a possibly incomplete space $E^p_{\cD,\psi}(\cV)$.

\begin{Def}\label{AMcRDef5.1}
Suppose that $\cD$ satisfies \eqref{H1}--\eqref{H3} on $L^2(\cV)$ for some $\omega\in[0,\pi/2)$ and $m \in\N$. For $\theta\in(\omega,\pi/2)$, $\psi\in\Psi(S^o_\theta)$ and $p\in[1,\infty]$, the space $E^p_{\cD,\psi}(\cV)$ consists of the set $\cS_{\psi}^\cD(T^p\cap T^{2})$ together with the seminorm
\[
\|u\|_{E^p_{\cD,\psi}} := \inf\{\|U\|_{T^p} : U\in T^p\cap T^{2} \text{ and } u=\cS_{\psi}^\cD U\}
\]
for all $u\in\cS_{\psi}^\cD(T^p\cap T^2)$.
\end{Def}

In \cite{AMcR}, the Hardy space $H^p_{\cD,\psi}(\cV)$ is defined to be an abstract completion of $E^p_{\cD,\psi}(\cV)$. Our question here is whether we can define $H^p_{\cD,\psi}(\cV)$ to be the completion of $E^p_{\cD,\psi}(\cV)$  in $L^p(\cV)$. So does the completion of $E^p_{\cD,\psi}(\cV)$ in $L^p(\cV)$ exist? This is immediate when \eqref{H2} holds and $p=2$, since for each $\theta \in (\omega,\pi/2)$ and nondegenerate $\psi\in\Psi(S^o_\theta)$, we have by~\eqref{eq:PfcRan}, \eqref{eq:QE} and Proposition~\ref{Prop: AMcR_Remark2.1} that $\cS^\cD_\psi(T^2)=\overline{\Ran(\cD)}$ with
\begin{equation}\label{eq:H2equiv}
\|u\|_{E^2_{\cD,\psi}} \eqsim \|\cQ_\psi^\cD u\|_{T^2} \eqsim \|u\|_2 \qquad \forall u\in\overline{\Ran(\cD)}.
\end{equation}
This motivates the following definition.

\begin{Def}
Suppose that $\cD$ satisfies \eqref{H1}--\eqref{H3} on $L^2(\cV)$ for some $\omega\in[0,\pi/2)$ and $m \in\N$. For each $\theta \in (\omega,\pi/2)$ and nondegenerate $\psi\in\Psi(S^o_\theta)$, let $H^2_{\cD,\psi}(\cV)$ denote the set $\overline{\Ran(\cD)}$ together with the norm $\|u\|_{H^2_{\cD,\psi}} := \|u\|_{E^2_{\cD,\psi}}$.
\end{Def}

When $p\in [1,2)$, we do not know whether or not the completion of $E^p_{\cD,\psi}(\cV)$ in $L^p(\cV)$ always exists, so we proceed under additional hypotheses on $\cD$. We begin by recording a routine extension of~\cite[Theorem~4.9 and Lemma 5.2]{AMcR}. In particular, the improved $\Psi(S^o_\theta)$ class exponents in the theorem below follow from~\eqref{eq:PfcOD} (for details, see~\cite[Proposition~7.5]{HvNP} or \cite[Theorem~6.2]{CMcM}).

\begin{Thm}\label{Thm: AMcR_4.9_5.2}
Suppose that $M$ is a doubling metric measure space satisfying~\eqref{D} and that $\cD$ satisfies \eqref{H1}--\eqref{H3} on $L^2(\cV)$ for some $\omega\in[0,\pi/2)$ and $m \in\N$. If $p\in[1,2]$, $\theta\in(\omega,\pi/2)$, $\beta>\kappa/2m$, $\varphi\in\Psi_{\beta}(S^o_\theta)$, $\psi\in\Psi_{\beta}(S^o_\theta)$ and $\tilde\psi\in\Psi^{\beta}(S^o_\theta)$, then
\begin{equation}\label{eq:AMcR4.9}
\|\cQ_{\tilde\psi}^\cD\cS_\svt{\psi}^\cD U\|_{T^p} \lesssim \|U\|_{T^p} \qquad \forall U\in T^p\cap T^2.
\end{equation}
If, in addition, all of $\varphi$, $\psi$ and $\tilde\psi$ are nondegenerate, then
\begin{equation}\label{eq:AMcR5.2Sets}
\cS_\varphi^\cD(T^p\cap T^2) = \cS_\psi^\cD(T^p\cap T^2) = \{u\in\overline{\Ran(\cD)} : \cQ_{\tilde\psi}^\cD u \in T^p\}
\end{equation}
with the norm equivalence
\begin{equation}\label{eq:AMcR5.2Norms}
\|u\|_{E^p_{\cD,\varphi}} \eqsim \|u\|_{E^p_{\cD,\psi}} \eqsim \|\cQ_{\tilde\psi}^\cD u\|_{T^p} \qquad \forall u\in E^p_{\cD,\varphi} = \cS_{\varphi}^\cD(T^p\cap T^{2}).
\end{equation}
Moreover, if the completion $H^p_{\cD,\varphi}$ of $E^p_{\cD,\varphi}$ in $L^p$ exists, and $H^p_{\cD,\varphi}\cap L^2=E^p_{\cD,\varphi}$, then there are unique extensions $\cS^\cD_{\psi}\in\mathcal{L}(T^p,H^p_{\cD,\varphi})$ and $\cQ^\cD_{\tilde\psi} \in \mathcal{L}(H^p_{\cD,\varphi},T^p)$, and $H^p_{\cD,\varphi} = \cS^\cD_{\psi}(T^p)$ with the norm equivalence
\begin{equation}\label{eq:HpEquivNorms}
\|u\|_{H^p_{\cD,\varphi}}
\eqsim 
\inf\{\|U\|_{T^p} : U\in T^p \text{ and } u=\cS^\cD_{\psi} U\}
\eqsim \|\cQ^\cD_{\tilde\psi}u\|_{T^p} \quad \forall u\in H^p_{\cD,\varphi}.
\end{equation}
\end{Thm}

\begin{proof}
In view of the remarks preceding the theorem, it remains to prove~\eqref{eq:HpEquivNorms}. It follows from \eqref{eq:AMcR5.2Sets} that $\|\cS^\cD_{\psi} U\|_{E^p_{\cD,\varphi}} \leq \|U\|_{T^p}$ for all $U\in T^p \cap T^2$, and so $\cS^\cD_{\psi}$ in $\mathcal{L}(T^2,L^2)$ extends by density to a unique operator in $\mathcal{L}(T^p,H^p_{\cD,\varphi})$. It follows from \eqref{eq:AMcR5.2Norms} that $\cQ^\cD_{\tilde\psi}$ in $\mathcal{L}(L^2,T^2)$ restricts to an operator in $\mathcal{L}(E^p_{\cD,\varphi},T^p)$, and since $E^p_{\cD,\varphi} = H^p_{\cD,\varphi} \cap L^2$, the density of $E^p_{\cD,\varphi}$ in $H^p_{\cD,\varphi}$ provides the unique extension of $\cQ^\cD_{\tilde\psi}$ in $\mathcal{L}(H^p_{\cD,\varphi},T^p)$. We then obtain  \eqref{eq:HpEquivNorms} by using the extended operators to appropriately extend~\eqref{eq:AMcR4.9}--\eqref{eq:AMcR5.2Norms}. This complete the proof.
\end{proof}

\begin{Rem}
In the context of Theorem~\ref{Thm: AMcR_4.9_5.2}, if the completion $H^p_{\cD,\psi}(\cV)$ of $E^p_{\cD,\psi}(\cV)$ in $L^p(\cV)$ exists, and $H^p_{\cD,\psi}(\cV)\cap L^2(\cV)=E^p_{\cD,\psi}(\cV)$, for \textit{some} nondegenerate $\psi\in\Psi_{\beta}(S^o_\theta)$, then \eqref{eq:HpEquivNorms} implies that these properties hold for \textit{all} nondegenerate $\psi\in\Psi_{\beta}(S^o_\theta)$. Therefore, we could adopt the notation in \cite{AMcR} whereby $H^p_{\cD}(\cV)$ denotes any of the equivalent Banach spaces $H^p_{\cD,\psi}(\cV)$. We found it convenient not to do this, however, given the technical nature of this article.
\end{Rem}

We now introduce atoms and molecules in order to show that $E^p_{\cD,\psi}(\cV) \subseteq L^p(\cV)$. 

\begin{Def}\label{moleculedef}
Suppose that $\cD$ satisfies \eqref{H1} and \eqref{H3} on $L^2(\cV)$ for some $m\in\N$. For $N\in\N$, a section $a\in L^2(\cV)$ is called an $H^1_\cD(\cV)$-$\textit{molecule of type }N$ when there exists a section $b \in \Dom(\cD^N)$ and a ball $B\subseteq M$ of radius $r(B)>0$ such that $a=\cD^N b$ and the following hold for all $k\in\N_0$:
\begin{enumerate}\setlength\itemsep{3pt}
\item $\|\ca_k(B)a\|_2 \leq 2^{-k} \mu(2^kB)^{-1/2}$;
\item $\|\ca_k(B)b\|_2 \leq r(B)^{m N} 2^{-k}\mu(2^kB)^{-1/2}$,
\end{enumerate}
where $\ca_0(B)=\ca_{B}$ and $\ca_k(B)=\ca_{2^{k}B\setminus 2^{k-1}B}$ for all $k\in\N$. An $H^1_\cD(\cV)$-$\textit{atom of type }N$ is defined in the same way, except that $a$ and $b$ are required to be supported on the ball $B$, which obviates (1) and (2) when $k\geq1$.
\end{Def}

The following proof uses a molecular characterisation obtained in~\cite[Section 6.1]{AMcR}.

\begin{Lem}\label{Lem: EpinLp}
Suppose that $M$ is a doubling metric measure space satisfying~\eqref{D} and that $\cD$ satisfies \eqref{H1}--\eqref{H3} on $L^2(\cV)$ for some $\omega\!\in\![0,\!\pi/2)$ and $m\in\N$. If $p\in[1,2]$, $\theta\in(\omega,\pi/2)$, $\beta>\kappa/2m$ and $\psi \in \Psi_{\beta}(S^o_\theta)$ is nondegenerate, then $E^p_{\cD,\psi}(\cV) \subseteq L^p(\cV)$.
\end{Lem}

\begin{proof}
When $p=2$, the result holds by~\eqref{eq:H2equiv}. When $p\in[1,2)$, it suffices to prove the result for a fixed nondegenerate $\psi$ in $\Psi_{\beta}(S^o_\theta)$ by \eqref{eq:AMcR5.2Norms}. Therefore, we fix $N\in\N$ and use the construction in \cite[Lemma~6.7]{AMcR} to fix a nondegenerate $\psi$ in $\Psi_{\beta}(S^o_\theta)$ such that $\cS^\cD_\psi (A)$ is an $H^1_\cD$-molecule of type $N$ whenever $A$ is a $T^1$-atom.

Now consider when $p=1$. For all $H^1_\cD$-molecules $a$ of type $N$, note that 
\begin{equation}\label{eq: L1uniform.mol}
\|a\|_1 \leq  \sum_{k=0}^\infty \mu(2^kB)^{{1}/{2}}\|\ca_k(B) a\|_2 \leq 2.
\end{equation}
Suppose that $u\in E^1_{\cD,\psi}$ and $V \in T^{1}\cap T^{2}$ such that $u=\cS^\cD_{\psi}V$ and $\|V\|_{T^1} \leq 2 \|u\|_{E^1_{\cD,\psi}}$. The atomic characterisation of $T^{1}$ in Theorem~\ref{maintentatomic} provides a sequence $(\lambda_j)_j$ in $\ell^1$ and a sequence $(A_j)_j$ of $T^{1}$-atoms such that $\sum_j\lambda_j A_j$ converges to $V$ in $T^{1}$ and $T^{2}$, and $\|(\lambda_j)_j\|_{\ell^1}\eqsim \|V\|_{T^1}$. The operator $\cS^\cD_{\psi}$ in $\mathcal{L}(T^2,L^2)$ is bounded from $(T^1\cap T^2, \|\cdot\|_{T^1})$ into $E^1_{\cD,\psi}$, by the definition of $E^1_{\cD,\psi}$, so $\sum_j \lambda_j \cS^\cD_{\psi} A_j$ converges to $u$ in $E^1_{\cD,\psi}$ and $L^2$. Now recall that $\psi$ has the property whereby each $\cS^\cD_{\psi} A_j$ is an $H^1_\cD$-molecule of type $N$, so in accordance with \eqref{eq: L1uniform.mol}, the sequence $(\cS^\cD_{\psi} A_j)_j$ is uniformly bounded in $L^1$, and as such, there exists $\tilde u$ in $L^1$ such that $\sum_j \lambda_j \cS^\cD_{\psi} A_j$ converges to $\tilde u$ in $L^1$. We must have $u =\tilde u \in L^1$, since $L^1$ and $L^2$ are embedded in $L^1_\text{loc}$, and so $\sum_j \lambda_j \cS^\cD_{\psi} A_j$ converges to $u$ in~$L^1$ with $\|u\|_1 = \lim_{n\rightarrow\infty} \|\sum_{j=1}^n \lambda_j \cS^\cD_{\psi} A_j\|_1 \lesssim \|(\lambda_j)_j\|_{\ell^1} \lesssim \|V\|_{T^1} \lesssim \|u\|_{E^1_{\cD,\psi}}$. This completes the proof when $p=1$.

Now consider when $p\in(1,2)$. We have shown that $E^1_{\cD,\psi} \subseteq L^1$, so by the definition of $E^1_{\cD,\psi}$, it follows that $\|\cS^\cD_{\psi} U\|_1 \lesssim \|\cS^\cD_{\psi} U\|_{E^1_{\cD,\psi}} \leq \|U\|_{T^1}$ for all $U\in T^1 \cap T^2$. Therefore, the operator $\cS^\cD_{\psi}$ in $\mathcal{L}(T^2,L^2)$ has an extension in $\mathcal{L}(T^1,L^1)$, and then by the interpolation of tent spaces in \eqref{TpInterp}, this extension is also in $\mathcal{L}(T^p,L^p)$. It follows that $E^p_{\cD,\psi} \subseteq L^p$, since for each $u\in E^p_{\cD,\psi}$, there exists $V \in T^{p}\cap T^{2}$ such that $u=\cS^\cD_{\psi}V$ and $\|V\|_{T^p} \leq 2 \|u\|_{E^p_{\cD,\psi}}$, hence $\|u\|_p = \|\cS^\cD_{\psi} V\|_p \lesssim \|V\|_{T^p} \lesssim \|u\|_{E^p_{\cD,\psi}}$.
\end{proof}

The proof of Lemma~\ref{Lem: EpinLp} shows that for each $N\in\N$ and $u\in E^1_{\cD,\psi}(\cV)$, there exists a sequence $(\lambda_j)_j$ in~$\ell^1$ and a sequence $(a_j)_j$ of $H^1_\cD(\cV)$-molecules of type $N$ such that $\sum_j\lambda_j a_j$ converges to $u$ in $E^1_{\cD,\psi}(\cV)$ and $L^1(\cV)$ with $\|(\lambda_j)_j\|_{\ell^1} \lesssim \|u\|_{E^1_{\cD,\psi}}$. Although this characterisation extends to completions of $E^1_{\cD,\psi}(\cV)$ (see Theorem~\ref{AMcRThm6.2}), it does not seem to guarantee that the completion of $E^1_{\cD,\psi}(\cV)$ in $L^1(\cV)$ exists. We introduce hypothesis \ref{SFinLq} on $\cD$ in the next theorem for this reason.

\begin{Thm}\label{Thm: H1injectivity} 
Suppose that $M$ is a doubling metric measure space satisfying~\eqref{D} and that $\cD$ satisfies \eqref{H1}--\eqref{H3} on $L^2(\cV)$ for some $\omega\in[0,\pi/2)$ and $m\in\N$. If $1\leq q\leq p \leq 2$, $\theta\in(\omega,\pi/2)$, $\beta>\kappa/2m$, $\psi\in\Psi_{\beta}(S^o_\theta)$ is nondegenerate and
\begin{equation}\tag*{(H4)$_{\Psi}$}\label{SFinLq}
\begin{minipage}[c]{0.85\textwidth}
there exists a nondegenerate function ${\tilde\psi} \in \Psi^{\beta}(S^o_\theta)$ such that the set \\
$\{F \in T^2\cap T^{q'} : \cS^{\cD^*}_{\tilde\psi^*} F \in L^{q'}(\cV)\}$ is weak-star dense in $T^{p'}(\cV_+)$,
\end{minipage}
\end{equation}
where $1/q+1/q'=1$, then the completion $H^p_{\cD,\psi}(\cV)$ of $E^p_{\cD,\psi}(\cV)$ in $L^p(\cV)$ exists. Moreover, it holds that $H^p_{\cD,\psi}(\cV)\cap L^2(\cV)=E^p_{\cD,\psi}(\cV)$.
\end{Thm}

\begin{proof}
Lemma~\ref{Lem: EpinLp} shows that $E^p_{\cD,\psi} \subseteq L^p$, so the existence of the completion of $E^p_{\cD,\psi}$ in $L^p$ will follow by proving (3) in Proposition~\eqref{Prop: relativecomp} with $X=E^p_{\cD,\psi}$ and $Y=L^p$. To this end, let $(u_n)_n$ denote a Cauchy sequence in $E^p_{\cD,\psi}$ that converges to 0 in $L^p$.  We claim that $(u_n)_n$ converges to 0 in $E^p_{\cD,\psi}$. To see this, fix ${\tilde\psi}$ in $\Psi^{\beta}(S^o_\theta)$ satisfying~\ref{SFinLq} so that $\mathscr{E}:=\{F \in T^2\cap T^{q'} : \cS^{\cD^*}_{\tilde\psi^*} F \in L^{q'}\}$ is weak-star dense in $T^{p'}$. For all $n\in\N$, we have by~\eqref{eq:AMcR5.2Norms} that
\begin{equation}\label{eq:embeq}
\|u_n\|_{E^p_{\cD,\psi}} \eqsim \|\cQ^\cD_{\tilde\psi} u_n\|_{T^p},
\end{equation}
and since $(u_n)_n$ is Cauchy in $E^p_{\cD,\psi}$, there exists $U$ in $T^{p}$ such that $\cQ^\cD_{\tilde\psi} u_n$ converges to $U$ in $T^{p}$. Using the duality pairing in~\eqref{TpDual}, for all $n\in\N$ and $F\in\mathscr{E}$, we have
\begin{align*}
|\langle U, F \rangle_{T^2}| 
&\leq |\langle U-\cQ^\cD_{\tilde\psi} u_n, F \rangle_{T^2}| + |\langle\cQ^\cD_{\tilde\psi} u_n, F\rangle_{T^2}| \\
&\lesssim \|U-\cQ^\cD_{\tilde\psi} u_n\|_{T^{p}} \|F\|_{T^{p'}} + \|u_n\|_{L^p}\|\cS^{\cD^*}_{\tilde\psi^*} F\|_{L^{p'}},
\end{align*}
since $2\leq p' \leq q'$ ensures that $L^2\cap L^{q'} \subseteq L^{p'}$ and $T^2\cap T^{q'} \subseteq T^{p'}$. Moreover, since $\|\cS^{\cD^*}_{\tilde\psi^*} F\|_{L^{p'}}<\infty$ and $\|F\|_{T^{p'}}<\infty$, the preceding convergence results imply that
\begin{equation}\label{T0}
\langle U,F\rangle_{T^2} = 0 \qquad  \forall F \in \mathscr{E}.
\end{equation}
Then, since $U\in T^p$ and $\mathscr{E}$ is weak-star dense in $T^{p'}$, it follows that $\langle U,F\rangle_{T^2} = 0$ for all $F\in T^{p'}$, hence $U=0$ and $(u_n)_n$ converges to 0 in $E^p_{\cD,\psi}$, as claimed. This proves that the completion $H^p_{\cD,\psi}$ of $E^p_{\cD,\psi}$ in $L^p$ exists.

The inclusion $E^p_{\cD,\psi} \subseteq H^p_{\cD,\psi} \cap L^2$ holds by~\eqref{eq:AMcR5.2Sets}. To prove the reverse inclusion, suppose that $u\in H^p_{\cD,\psi} \cap L^2$. The density of $E^p_{\cD,\psi}$ in $H^p_{\cD,\psi}$ provides a sequence $(u_n)_n$ in $E^p_{\cD,\psi}$ that converges to $u$ in $H^p_{\cD,\psi}$. This sequence also converges in  $L^p$ because $H^p_{\cD,\psi}\subseteq L^p$, and as in the previous paragraph, there exists $U$ in $T^p$ such that $\cQ^\cD_{\tilde\psi} u_n$ converges to $U$ in $T^{p}$. For all $n\in\N$ and $F\in\mathscr{E}$, we have
\begin{align*}
|\langle U-\cQ^\cD_{\tilde\psi} u, F \rangle_{T^2}| 
&\leq |\langle U-\cQ^\cD_{\tilde\psi} u_n, F \rangle_{T^2}| + |\langle\cQ^\cD_{\tilde\psi} u_n-\cQ^\cD_{\tilde\psi} u, F\rangle_{T^2}| \\
&\lesssim \|U-\cQ^\cD_{\tilde\psi} u_n\|_{T^{p}} \|F\|_{T^{p'}} + \|u_n-u\|_{p}\|\cS^{\cD^*}_{\tilde\psi^*} F\|_{p'}.
\end{align*}
The preceding convergence arguments then show that $U=\cQ^\cD_{\tilde\psi} u\in T^p\cap T^2$, and since $\|\cQ^\cD_{\tilde\psi} u\|_{T^p} \eqsim \|u\|_{E^p_{\cD,\psi}}$, we conclude that $u\in E^p_{\cD,\psi}$, as required.
\end{proof}

\begin{Rem}\label{Rem: Tc}
Note that \ref{SFinLq} holds whenever $\cS^{\cD^*}_{\tilde\psi^*}(T^2_c(\cV_+)) \subseteq L^{q'}(\cV)$, where $T^2_c(\cV_+)$ denotes the space of compactly supported sections in $T^2(\cV_+)$. This is because $T^2_c(\cV_+)$ is weak-star dense in $T^{p'}(\cV_+)$ for all $p\in[1,2]$. To see this, let $(K_n)_n$ denote an increasing sequence of compact sets that exhaust $M\times\R_+$. For all $F\in T^p(\cV_+)$ and $G\in T^{p'}(\cV_+)$, we have $\int_0^\infty\int_M |\langle F_t(x),G_t(x) \rangle_x|\, d\mu(x)dt/t \lesssim \|F\|_{T^p}\|G\|_{T^{p'}}$ by duality (see \eqref{TpDual}). The dominated convergence theorem then implies that $\langle F, \mathbf{1}_{K_n}G \rangle_{T^2}$ converges to $\langle F,G \rangle_{T^2}$, which proves the weak-star density, since $\mathbf{1}_{K_n}G\in T^2_c(\cV_+)$.
\end{Rem}

\subsection{Molecular Theory}\label{section: Molecules}
We defined $H^1_\cD(\cV)$-molecules and atoms in Definition~\ref{moleculedef}. The molecular characterisation of $H^1_{\cD,\psi}(\cV)$ below is based on the characterisation obtained in~\cite[Theorem 6.2]{AMcR}. It is convenient to first introduce the following spaces.

\begin{Def}\label{AMcRDef6.1}
Suppose that $\cD$ satisfies \eqref{H1} and \eqref{H3} on $L^2(\cV)$ for some $m\in\N$. For $N\in\N$, the Banach space $H^1_{\cD,\text{mol}(N)}(\cV)$ is the set of all $u$ in $L^1(\cV)$ for which there exist a sequence $(\lambda_j)_j$ in $\ell^1$ and a sequence $(a_j)_j$ of $H^1_\cD$-molecules of type $N$ such that $\sum_j\lambda_j a_j$ converges to $u$ in $L^1(\cV)$, together with the norm
\[
\|u\|_{H^1_{\cD,\text{mol}(N)}} := \inf \{\|(\lambda_j)_j\|_{\ell^1} : \textstyle{\sum_j}\lambda_j a_j\text{ converges to $u$ in $L^1$}\}.
\]
The Banach space $H^1_{\cD,\text{at}(N)}(\cV)$ is defined by replacing molecules with atoms.
\end{Def}

The $L^1(\cV)$ convergence required in the above definition ensures that $H^1_{\cD,\text{mol}(N)}(\cV)$ and $H^1_{\cD,\text{at}(N)}(\cV)$ are complete. This is because molecules and atoms are uniformly bounded in $L^1(\cV)$. In particular, if $(u_n)_n$ is a sequence in $H^1_{\cD,\text{mol}(N)}(\cV)$ such that $\sum_n \|u_n\|_{H^1_{\cD,\text{mol}(N)}}$ is finite, then the uniform $L^1(\cV)$ bound for molecules and the dominated convergence theorem imply that $\sum_n u_n$ converges in the $H^1_{\cD,\text{mol}(N)}(\cV)$ norm to some $u\in H^1_{\cD,\text{mol}(N)}(\cV)$, hence $H^1_{\cD,\text{mol}(N)}(\cV)$ is complete. The $L^1(\cV)$ convergence requirement also distinguishes these spaces from those in the literature that are defined as an abstract completion of a molecular or atomic space on which $L^2(\cV)$ convergence is required. This is discussed further in Remark~\ref{Rem: Emol}.

The embedding $H^1_{\cD,\psi}(\cV) \subseteq L^1(\cV)$ is not required to define the molecular space nor the atomic space, since ${H^1_{\cD,\text{at}(N)}(\cV)\subseteq H^1_{\cD,\text{mol}(N)}(\cV) \subseteq L^1(\cV)}$ is automatic. It is only when the embedding of $H^1_{\cD,\psi}(\cV)$ in $L^1(\cV)$ holds, however, that we can establish the following connection.

\begin{Thm}\label{AMcRThm6.2}
Suppose that $M$ is a doubling metric measure space satisfying \eqref{D} and that $\cD$ satisfies \eqref{H1}--\eqref{H3} on $L^2(\cV)$ for some $\omega\in[0,\pi/2)$ and $m \in \N$. Also, assume that for some $\theta\in(\omega,\pi/2)$, $\beta>\kappa/2m$ and nondegenerate $\psi\in\Psi_{\beta}(S^o_\theta)$, the completion $H^1_{\cD,\psi}(\cV)$ of $E^1_{\cD,\psi}(\cV)$ in $L^1(\cV)$ exists, and $H^1_{\cD,\psi}(\cV)\cap L^2(\cV)=E^1_{\cD,\psi}(\cV)$. It follows that if $N\in\N$ and $N>\kappa/2m$, then $H^1_{\cD,\psi}(\cV)=H^1_{\cD,\text{mol}(N)}(\cV)$.
\end{Thm}

\begin{proof}
Suppose that $N\in\N$. The proof that $H^1_{\cD,\psi}\subseteq H^1_{\cD,\text{mol}(N)}$ follows that of Lemma~\ref{Lem: EpinLp}, except we need to replace $L^2$ convergence with $H^1_{\cD,\psi}$ convergence. We use the construction in \cite[Lemma~6.7]{AMcR} to fix a nondegenerate ${\tilde\psi}$ in $\Psi_\beta(S^o_\theta)$ such that $\cS^\cD_{\tilde\psi}A$ is an $H^1_\cD$-molecule of type $N$ whenever $A$ is a $T^1$-atom. Suppose that $u\in H^1_{\cD,\psi}$ and use \eqref{eq:HpEquivNorms} to choose $V$ in $T^1$ such that $u=\cS^\cD_{\tilde\psi}V$ and $\|V\|_{T^1} \leq 2 \|u\|_{H^1_{\cD,\psi}}$. The atomic characterisation of $T^{1}$ in Theorem~\ref{maintentatomic} provides a sequence $(\lambda_j)_j$ in $\ell^1$ and a sequence $(A_j)_j$ of $T^{1}$-atoms such that $\sum_j\lambda_j A_j$ converges to $V$ in $T^{1}$ and $\|(\lambda_j)_j\|_{\ell^1} \lesssim \|V\|_{T^1}$. It follows that $\sum_j \lambda_j \cS^\cD_{\tilde\psi} A_j$ converges to $u$ in $H^1_{\cD,\psi}$ and in $L^1$ because $\cS^\cD_{\tilde\psi} \in \mathcal{L}(T^1,H^1_{\cD,\psi})$ by \eqref{eq:HpEquivNorms} and $H^1_{\cD,\psi}\subseteq L^1$. Now recall that ${\tilde\psi}$ has the property whereby each $\cS^\cD_{\tilde\psi}A_j$ is an $H^1_\cD$-molecule of type $N$, so then $u\in H^1_{\cD,\text{mol}(N)}$ and $\|u\|_{H^1_{\cD,\text{mol}(N)}} \leq \|(\lambda_j)_j\|_{\ell^1} \lesssim \|V\|_{T^1} \lesssim \|u\|_{H^1_{\cD,\psi}}$, hence $H^1_{\cD,\psi}\subseteq H^1_{\cD,\text{mol}(N)}$.

Now suppose that $N\in\N$, $N>\kappa/2m$ and $u\in H^1_{\cD,\text{mol}(N)}$. Then $u\in L^1$ and there is a sequence $(\lambda_j)_j$ in $\ell^1$ and a sequence $(a_j)_j$ of $H^1_\cD$-molecules of type $N$ such that $\sum_j\lambda_j a_j$ converges to $u$ in $L^1$ with $\|(\lambda_j)_j\|_{\ell^1} \leq 2 \|u\|_{H^1_{\cD,\text{mol}(N)}}$. The construction in \cite[Lemma~6.8]{AMcR} allows us to fix ${\tilde{\tilde\psi}}$ in $\Psi^{\beta}(S^o_\theta)$ such that $\cQ^\cD_{{\tilde{\tilde\psi}}}$ is uniformly bounded in $T^1$ on all $H^1_\cD$-molecules of type $N$ (this requires $N>\kappa/2m$), so by~\eqref{eq:HpEquivNorms} we have
\[
\left\|\sum_{j=1}^l \lambda_ja_j - \sum_{j=1}^k \lambda_ja_j\right\|_{H^1_{\cD,\psi}}
\lesssim \sum_{j=k+1}^l |\lambda_j| \|\cQ^\cD_{{\tilde{\tilde\psi}}}a_j\|_{T^1} 
\lesssim \sum_{j=k+1}^l |\lambda_j|
\]
whenever $l>k>0$. Therefore, there exists $v$ in $H^1_{\cD,\psi}$ such that $\sum_j \lambda_ja_j$ converges to $v$ in $H^1_{\cD,\psi}$, and hence in $L^1$ because $H^1_{\cD,\psi}\subseteq L^1$. It follows that $u=v\in H^1_{\cD,\psi}$ with
\[
\|u\|_{H^1_{\cD,\psi}} \lesssim \lim_{k\rightarrow\infty} \sum_{j=1}^k |\lambda_j| \|\cQ^\cD_{{\tilde{\tilde\psi}}}a_j\|_{T^1}
\lesssim \|(\lambda_j)_j\|_{\ell^1}\lesssim \|u\|_{H^1_{\cD,\text{mol}(N)}},
\]
so $H^1_{\cD,\text{mol}(N)} \subseteq H^1_{\cD,\psi}$ and the proof is complete.
\end{proof}

\begin{Rem}
The proof of Theorem~\ref{AMcRThm6.2} shows that the same result holds when the $L^1(\cV)$ convergence required in Definition~\ref{AMcRDef6.1} is replaced with $H^1_{\cD,\psi}(\cV)$ convergence.
\end{Rem}

\begin{Rem}\label{Rem: Emol}
If we define $E^1_{\cD,\text{mol}(N)}(\cV)$ to be the normed space obtained by replacing $L^1(\cV)$ convergence with $L^2(\cV)$ convergence in Definition~\ref{AMcRDef6.1}, then we can prove that $E^1_{\cD,\psi}(\cV) = E^1_{\cD,\text{mol}(N)}(\cV)$ without assuming  that the embedding $H^1_{\cD,\psi}(\cV) \subseteq L^1(\cV)$ holds. This was known previously (see~\cite[Theorem 3.5]{HMMc}). In particular, the proof of Lemma~\ref{Lem: EpinLp} shows that $E^1_{\cD,\psi}(\cV) \subseteq  E^1_{\cD,\text{mol}(N)}(\cV)$, whilst the reverse inclusion is proved in a manner similar to that of Theorem~\ref{AMcRThm6.2}. This means that we can identify any completion of $E^1_{\cD,\psi}(\cV)$ with any completion of $ E^1_{\cD,\text{mol}(N)}(\cV)$, but both are still abstract spaces and it is not known whether either can be embedded in $L^1(\cV)$, or in any function space, without the extra hypotheses on $\cD$ in Theorem~\ref{Thm: H1injectivity} (or Theorem~\ref{Thm: H1injectivityFPS}).
\end{Rem}

\subsection{The Embedding $H^p_{L}\subseteq L^p$ for Divergence Form Elliptic Operators}\label{section: ApplicationsI}
It is a simple matter to verify the hypotheses of Theorem~\ref{Thm: H1injectivity} for an operator that generates a semigroup satisfying pointwise kernel estimates. We demonstrate this by obtaining Theorem~\ref{Thm: HpinLpIntroL} as a special case of the more general result below.

Let $M=\R^n$ and consider the divergence form operator $L = -\Div A \nabla$ acting on $L^2(\R^n)$ and interpreted in the usual weak sense via a sesquilinear form, where $A \in L^\infty(\R^n,\cL(\C^n))$ is elliptic in the sense that there exists $\lambda>0$ such that 
\[
\re \langle A(x)\zeta,\zeta\rangle_{\C^n}\geq\lambda|\zeta|^2 \qquad \forall \zeta\in\C^n, \ \text{ a.e. } x\in\R^n
\]
There exists $\omega_L\in[0,\pi/2)$, depending on $\lambda$ and $\|A\|_\infty$, such that $L$ is $\omega_L$-sectorial (see, for instance,~\cite[Chapter 2]{AuscherMem2007}), hence $L:\Dom(L)\subseteq L^2(\R^n) \rightarrow L^2(\R^n)$ satisfies \eqref{H1}--\eqref{H2} with $\omega = \omega_L$. Note that $\Dom(L) = \{u\in W^{1,2}(\R^n) : A\nabla u\in \Dom(\nabla^*)\}$. It is also known that $L$ satisfies \eqref{H3} with $m=2$ (see \cite[Lemma~2.1]{AHLMT}).

In order to embed $H^p_{L,\psi}(\R^n)$ in $L^p(\R^n)$ when $1\leq q \leq p \leq 2$ and $1/q' + 1/q = 1$, we assume that there exists $g\in L^{2}_{\text{loc}}((0,\infty))$ such that the analytic semigroup $(e^{-tL^*})_{t>0}$ generated by the adjoint $-L^*$ on $L^2(\R^n)$ satisfies
\begin{equation}\label{GnewLq}
\|e^{-tL^*}u\|_{q'} \leq g(t) \|u\|_2 \qquad \forall u\in L^2(\R^n).
\end{equation}
This assumption is always satisfied when $2n / (n+2) \leq q \leq 2$ in dimension $n\geq3$  (see \cite[Proposition~3.2]{AuscherMem2007} and \cite[Lemma~2.25]{HMMc}). It remains an open question, however, as to whether the following theorem holds in the absence of estimates such as~\eqref{GnewLq}. 

\begin{Thm}\label{Thm: HpinLpLNEW}
Suppose that $A \in L^\infty(\R^n,\cL(\C^n))$ is elliptic and $L=-\Div A \nabla$ on $L^2(\R^n)$ satisfies \eqref{GnewLq} for some $q \in [1,2]$. If $q\leq p\leq 2$, $\theta\in(\omega_L,\pi/2)$, $\beta>n/4$ and $\psi\in\Psi_\beta(S^o_\theta)$ is nondegenerate, then the completion $H^p_{L,\psi}(\R^n)$ of $E^p_{L,\psi}(\R^n)$ in $L^p(\R^n)$ exists. Moreover, if $q=1$, $N\in\N$ and $N>n/4$, then $H^1_{L,\psi}(\R^n)=H^1_{L,\text{mol}(N)}(\R^n)$, and when $A$ is self-adjoint, then also $H^1_{L,\psi}(\R^n)=H^1_{L,\text{at}(N)}(\R^n)$.
\end{Thm}

\begin{proof}
We will use \eqref{GnewLq} to show that~\ref{SFinLq} holds with $\kappa = n$. The hypotheses of Theorem~\ref{Thm: H1injectivity} will then be satisfied, since it was noted above that $L$ satisfies \eqref{H1}--\eqref{H3} with $\omega = \omega_L$ and $m=2$. To this end, choose $\theta\in(\omega_L,\pi/2)$, define the nondegenerate function ${\tilde\psi}(z)=ze^{-z}$ on $S^o_\theta \cup \{0\}$ and note that $\tilde\psi \in \Psi^{\beta}(S^o_\theta)$ for any $\beta>n/4$. Let $T^2_c(\R^{n+1}_+)$ denote the space of compactly supported functions in $T^2(\R^{n+1}_+)$. For each $F\in T^2_c(\R^{n+1}_+)$, there is a ball $B \subseteq M$ and $r>1$ such that $\sppt (F) \subseteq B \times [1/r,r]$, and so we have
\begin{align}\begin{split}\label{H4est}
\|\cS^{L^*}_{\tilde\psi^*}F\|_{q'}
&= \left\| \int_{1/r}^r t^2L^*e^{-t^2L^*}F_t \,\frac{dt}{t}\right\|_{q'} \\
&\leq \int_{1/r}^r \|e^{-(t^2/2)L^*}t^2L^*e^{-(t^2/2)L^*}F_t\|_{q'} \,\frac{dt}{t} \\
&\leq \int_{1/r}^r g(t^2/2)\|t^2L^*e^{-(t^2/2)L^*}F_t\|_2 \,\frac{dt}{t} \\
&\lesssim \left(\int_{1/r}^r (g(t^2/2))^2 \,\frac{dt}{t} \right)^{1/2} \left(\int_{0}^\infty \|F_t\|_2^2 \,\frac{dt}{t} \right)^{1/2}\\
&\lesssim \|F\|_{T^2},
\end{split}\end{align}
where the third line uses~\eqref{GnewLq}, and the fourth line uses the analyticity of the semigroup $(e^{-tL^*})_{t>0}$ (see, for instance, \cite[Theorem~II.4.6]{EngelNagel2000}) followed by the Cauchy--Schwarz inequality. This shows that $\cS^{L^*}_{\tilde\psi^*} (T^{2}_c(\R^{n+1}_+)) \subseteq L^{q'}(\R^n)$, so  Remark~\ref{Rem: Tc} implies that~\ref{SFinLq} holds with $\kappa=n$, as required.

We have now shown that the hypotheses of Theorem~\ref{Thm: H1injectivity} hold. Moreover, when $q=1$, the hypotheses of Theorem~\ref{AMcRThm6.2} follow. The conclusions of those two theorems complete the proof, except for the atomic characterisation in the case when $q=1$ and $A$ is self-adjoint, but then $L=-\Div A \nabla$ satisfies the requirements of Theorem~\ref{Thm: HpinLPforHLMMY} (see \cite[Proposition~3.2]{AuscherMem2007} for a proof of the Davies--Gaffney estimates~\eqref{DG}), so we refer the reader to the proof of that theorem in Section~\ref{section: HLMMY}.
\end{proof}

Theorem~\ref{Thm: HpinLpIntroL} is a special case of the above result.

\begin{proof}[Proof of Theorem~\ref{Thm: HpinLpIntroL}]
This is a special case of Theorem~\ref{Thm: HpinLpLNEW}, since property \eqref{Gnew} corresponds to property \eqref{GnewLq} with $q=1$.
\end{proof}

\section{Self-Adjoint Operators with Finite Propagation Speed}\label{section: FPS}
We now restrict the theory of the previous section to the context of any self-adjoint operator $D: \Dom(D) \subseteq L^2(\cV) \rightarrow L^2(\cV)$ for which the associated unitary $C_0$-group $(e^{itD})_{t\in\R}$ has finite propagation speed. The existence of this group is guaranteed by Stone's Theorem because $D$ is self-adjoint. The defining features of such a group are that the mapping $t\mapsto e^{itD}$ is strongly continuous from $\R$ to $\mathcal{L}(L^2(\cV))$ with $e^{i(s+t)D}=e^{isD}e^{itD}$, $e^{itD}|_{t=0}=I$ and $\frac{d}{dt}(e^{itD}u)|_{t=0}=iDu$ for all $u\in \Dom(D)=\{u\in L^2(\cV) : \frac{d}{dt}(e^{itD}u)|_{t=0} \text{ exists in } L^2(\cV)\}$. An introduction to the theory of such groups can be found in~\cite{Kato,EngelNagel2000}. The group $(e^{itD})_{t\in\R}$ is said to have \textit{finite propagation speed} when there exists a finite constant $c_D>0$ such that for all $u\in L^2(\mathcal{V})$ satisfying $\sppt(u)\subseteq F\subseteq M$ and all $t\in \R$, it holds that $\sppt(e^{itD}u)\subseteq \{x\in M : \rho(\{x\},F)\leq c_D |t|\}$.  We begin by establishing that these assumptions allow us to apply the theory from the previous section with $\cD=D$, $\omega=0$ and $m=1$.

\begin{Lem}\label{Lem: ODestRes}
If $D$ is a self-adjoint operator on $L^2(\cV)$ and the group $(e^{itD})_{t\in\R}$ has finite propagation speed $c_D>0$, then $D$ satisfies \eqref{H1}--\eqref{H3} with $\omega=0$ and $m=1$. 
\end{Lem}

\begin{proof}
Since $D$ is self-adjoint, it satisfies \eqref{H1} and \eqref{H2} with $\omega=0$, $C_\theta=1/\sin\theta$ and $c_\theta=1$. It remains to prove \eqref{H3}. Let $E$ and $F$ denote measurable subsets of $M$. The finite propagation speed implies that $\ca_E e^{itD} \ca_F=0$ whenever $\rho(E,F)> c_D |t|$. For all $z\in\C$ with $\im (\pm z) > 0$, we use the integral representation of the resolvent
$
(zI-D)^{-1}=\mp i\int_0^\infty e^{\pm izt}e^{\mp itD}\, dt
$
to obtain
\begin{align*}
\|\ca_E (zI-D)^{-1} \ca_F\| 
\leq \int_{\rho(E,F)/{c_D}}^\infty |e^{\pm izt}| \|\ca_E  e^{\mp itD } \ca_F\|\, d t
\leq \int_{{\rho(E,F)}/{{c_D}}}^\infty e^{-(\im(\pm z))t} \, dt.
\end{align*}
For each $\theta\in(0,\pi/2)$, it follows that
\[
\|\ca_E (zI-D)^{-1} \ca_F\| 
\leq \frac{C_{\theta}}{|z|} \exp\left(-\frac{\rho(E,F)|z|}{{c_D} C_{\theta}}\right) \qquad \forall z\in \C\setminus S_{\theta},
\]
which implies \eqref{H3} with $m=1$.
\end{proof}

The algebra of complex-valued bounded Borel measurable functions on $\R$ is denoted by $B^\infty(\R)$. The Spectral Theorem for self-adjoint operators provides $D$ with a bounded $B^\infty(\R)$ functional calculus such that $\|f(D)\| \leq \|f\|_\infty$ for all $f\in B^\infty(\R)$. This coincides with the holomorphic functional calculus defined by~\eqref{eq:Pfc} and~\eqref{eq:Hfc} when $f \in H^\infty(S^o_\theta\cup\{0\})$ because the holomorphic functional calculus is unique with respect to \eqref{IdH}--\eqref{cvH}. In particular, it is well known (see~\cite[Chapter~XX, \S1]{Lang1993}) that the Borel functional calculus is an algebra homomorphism from $B^\infty(\R)$ into $\mathcal{L}(L^2(\cV))$ that satisfies~\eqref{IdH} and \eqref{ResH}, with $\R$ in place of $S^o_\theta\cup\{0\}$, as well as the following convergence lemma, which is related to~\eqref{cvH}:
\begin{align}
\label{cvB}\textrm{\begin{minipage}[c]{0.9\textwidth}if $(f_n)_n$ is a sequence in $B^\infty(\R)$ that converges pointwise to a function $f$ in $B^\infty(\R)$, and $\textstyle \sup_n \|f_n\|_{\infty} < \infty$, then $\textstyle \lim_n f_n(D)u = f(D)u$ for all $u\in L^2(\cV)$.\end{minipage}}
\end{align}
The  orthogonal decomposition $L^2(\cV) = \overline{\Ran(D)} {\overset\perp\oplus} N(D)$ and the properties of the Borel functional calculus allow us to prove the following Calder\'{o}n reproducing formula.

\begin{Prop}\label{Prop: Calderon}
Suppose that $D$ is self-adjoint on $L^2(\cV)$. If $f$ and $g$ in $B^\infty(\R)$ satisfy $f(0)g(0)=0$, $\int_0^\infty |f(\pm t)g(\pm t)|\,\frac{d t}{t} <\infty$ and $\int_0^\infty f(\pm t) g(\pm t)\,\frac{d t}{t}=1$, then
\begin{equation}\label{eqCglo}
\int_0^\infty f_t(D)g_t(D)u\,\frac{d t}{t} = \mathsf{P}_{\overline{\Ran(D)}}\,u \qquad \forall u\in L^2(\cV),
\end{equation}
where $\mathsf{P}_{\overline{\Ran(D)}}$ denotes the projection from $L^2(\cV)$ onto $\overline{\Ran(D)}$.
\end{Prop}

\begin{proof}
Suppose that $f$ and $g$ in $B^\infty(\R)$ satisfy the hypotheses of the proposition. For each $n\in\N$, we have
\[
h_n(x):= \int_{\tfrac{1}{n}}^n f_t(x)g_t(x)\,\frac{d t}{t}
=\begin{cases}
\int_{x/n}^{xn} f( t)g( t)\,\frac{d t}{t}, & \text{if } x>0; \\
0, & \text{if }x=0; \\
\int_{|x|/n}^{|x|n} f(-t)g(-t)\,\frac{d t}{t}, & \text{if } x<0.
\end{cases}
\]
The sequence $(h_n)_n$ converges pointwise on $\R$ to the characteristic function $\mathbf{1}_{\R\setminus\{0\}}$, and $\sup_n \|h_n\|_\infty \leq \int_0^\infty |f(\pm t)g(\pm t)|\,\frac{d t}{t} < \infty$, so it follows from~\eqref{cvB} that
\[
\int_0^\infty f_{t}(D)g_{t}(D)u\,\frac{d t}{t} = \lim_{n\rightarrow\infty} h_n(D)u = \mathbf{1}_{\R\setminus\{0\}} (D)u = \mathsf{P}_{\overline{\Ran(D)}}\, u \qquad \forall u\in L^2(\cV),
\]
where the final equality relies on the fact that $\mathbf{1}_{\R}(D) = I$ and $\mathbf{1}_{\{0\}}(D) = \mathsf{P}_{\Nul(D)}$.
\end{proof}

We now require a class of functions that interact well with finite propagation speed. To this end, a function on $\R$ is called \textit{nondegenerate} when it is not identically zero on $(0,\infty)$ nor on $(-\infty,0)$. The Fourier transform of any Schwartz function $f\in\cS(\R)$ is denoted by $\widehat{f}$. For  $\delta>0$ and $N\in\N$, define
\begin{align*}
\widetilde{\Theta}^\delta(\R) &= \{\varphi \in \cS(\R) : \sppt \widehat{\varphi} \subseteq [-\delta,\delta]\},\\
\widetilde{\Psi}_N^\delta(\R) &= \{\eta \in \widetilde{\Theta}^\delta(\R) : \partial^{k-1}\eta(0)=0 \text{ for all } k\in\{1,\ldots, N\}\},
\end{align*}
$\widetilde{\Theta}(\R)=\bigcup_{\delta>0}\widetilde{\Theta}^\delta(\R)$, $\widetilde{\Psi}_N(\R) = \bigcup_{\delta>0} \widetilde{\Psi}_N^\delta(\R)$ and $\widetilde{\Psi}(\R) = \widetilde{\Psi}_1(\R)$. For $\varphi\in\widetilde\Theta(\R)$, the Fourier inversion formula and the $B^\infty(\R)$ functional calculus imply that
\begin{equation}\label{FT}
\varphi (D )u = \frac{1}{2\pi} \int_{\R} \widehat{\varphi}(t) e^{itD}u\, d t \qquad \forall u\in L^2(\cV).
\end{equation}
For $\eta\in\widetilde{\Psi}(\R)$, using the $B^\infty(\R)$ functional calculus, define $\cQ _\eta^D$ in $\mathcal{L}(L^2,T^2)$  by $(\cQ _\eta^Du)_t = \eta(tD)u$, and $\cS_\eta^D$ in $\mathcal{L}(T^2,L^2)$  by $\cS_\eta^DU=\int^\infty_0\eta(sD) U_s\frac{ds}{s}$, as well as the space $E^p_{D,\eta}(\cV) = \cS_{\eta}^\cD(T^p\cap T^{2})$.  This extends Definitions~\ref{Def: QandS} and~\ref{AMcRDef5.1}, which use the $H^\infty(S^o_\theta\cup\{0\})$ functional calculus. Also, note that
\begin{equation}\label{eq:BPfcRan}
\eta(D)u = \mathsf{P}_{\overline{\Ran(D)}}\, \eta(D)\, \mathsf{P}_{\overline{\Ran(D)}}\, u \qquad \forall u\in L^2(\cV),
\end{equation}
since $\eta(0)=0$ and $\mathbf{1}_{\{0\}}(D) = \mathsf{P}_{\Nul(D)}$.

The following corollary of Proposition~\ref{Prop: Calderon} extends the Calder\'{o}n reproducing formula in Proposition~\ref{Prop: AMcR_Remark2.1} and allows us to incorporate $\widetilde\Psi(\R)$ class functions into the theory of Section~\ref{section: Pointwise}.

\begin{Cor}\label{Cor: Repr}
Suppose that $D$ is a self-adjoint operator on $L^2(\cV)$. If $\sigma,\tau>0$, $\theta\in(0,\pi/2)$ and $\eta\in\widetilde\Psi(\R)$ is nondegenerate, then there exists a nondegenerate $\psi\in\Psi_\sigma^\tau(S_\theta^o)$ such that $\cS_\psi^D\cQ_{\eta}^D u = \cS_{\eta}^D\cQ_\psi^D u = \mathsf{P}_{\overline{\Ran(D)}}\,u$ for all $u\in L^2(\cV)$.
\end{Cor}

\begin{proof}
Suppose that $\eta\in\widetilde\Psi^\delta_N(\R)$ for some $\delta>0$ and $N\in\N$. It follows by the Paley--Wiener Theorem that $\eta$ extends to an entire function satisfying $|\eta(z)|\leq C e^{\delta|z|}$ for some constant $C>0$ and all $z\in\C$. Now consider $\sigma,\tau>0$ and $\theta\in(0,\pi/2)$. When $\re(z)>0$, define $\psi(z) = \alpha_+ z^\sigma e^{-2\delta z\sec\theta}\eta^*(z)$, and when $\re(z)<0$, define $\psi(z) = \alpha_-(- z)^\sigma e^{2\delta z\sec\theta}\eta^*(z)$, where $\alpha_\pm$ are the normalising constants defined by
\[
\alpha_+\int_0^\infty t^\sigma e^{-2\delta t\sec\theta}|\eta(t)|^2 \, \frac{dt}{t} = 1 \qquad\text{and}\qquad
\alpha_-\int_0^\infty t^\sigma e^{-2\delta t\sec\theta}|\eta(-t)|^2 \, \frac{dt}{t} = 1.
\]
The integrals above are positive, so the normalising constants exist, and for all $z\in \C$ with $\re(z)\neq 0$, we have
\[
\int_0^\infty \psi(tz)\eta(tz) \, \frac{dt}{t} = 1.
\]
Finally, define $\psi(0)=0$ so that $\psi\in\Psi_\sigma^\tau(S_\theta^o)$, and since $\psi$ is clearly nondegenerate, the result follows from Proposition~\ref{Prop: Calderon}.
\end{proof}

The next result shows how $\widetilde\Theta(\R)$ functions interact with finite propagation speed.   In particular, the off-diagonal estimate in~\eqref{mn} is much sharper than that in~\eqref{eq:PfcOD}.

\begin{Lem}\label{Lem: ODestSchwartz}
Suppose that $D$ is a self-adjoint operator on $L^2(\cV)$ and the group $(e^{itD})_{t\in\R}$ has finite propagation speed $c_D>0$. If $\delta>0$ and $\varphi\in\widetilde{\Theta}^\delta(\R)$, then
\begin{equation}\label{mn}
\|\ca_E \varphi_t(D ) \ca_F\| \leq \tfrac{1}{\pi} \|\widehat{\varphi}\|_{\infty} \max\left\{\delta-\frac{\rho(E,F)}{c_Dt},0\right\}
\leq C e^{-\rho(E,F)/t}
\end{equation}
for all $t>0$, all measurable sets $E, F\subseteq M$, and some $C>0$. 
\end{Lem}

\begin{proof}
Suppose that $\varphi\in\widetilde{\Theta}(\R)$ with $\sppt \widehat\varphi \subseteq [-\delta,\delta]$. It follows from~\eqref{FT} that
\[
\varphi_t(D )u= \frac{1}{2\pi} \int_{-\infty}^\infty \widehat{\varphi_t}(s) e^{isD}u\, d s=\frac{1}{2\pi} \int_{|s|\leq \delta t} \widehat{\varphi}\left(\frac{s}{t}\right) e^{isD}u \ \frac{d s}{t}\quad \forall t>0, \ \forall u \in L^2(\cV).
\]
Suppose that $E$ and $F$ are measurable subsets of $M$. The finite propagation speed implies $\ca_E e^{isD} \ca_F=0$ whenever $\rho(E,F)> {c_D} |s|$, hence $\ca_E \varphi_t(D) \ca_F=0$ whenever $\rho(E,F)/{c_D}>\delta t$ by the preceding formula. In addition, if $\rho(E,F) / {c_D}\leq \delta t$, then
\begin{align*}
\|\ca_E\varphi_t(D ) \ca_F\| &\leq \frac{1}{2\pi} \int_{\rho(E,F)/{c_D} \leq |s| \leq \delta t} \left|\widehat{\varphi}\left(\frac{s}{t}\right)\right| \|\ca_E  e^{isD } \ca_F\|\ \frac{d s}{t}\\
&\leq \frac{1}{2\pi} \int_{{\rho(E,F)}/{{c_D} t} \leq |\sigma| \leq \delta} |\widehat{\varphi}(\sigma)| \, d\sigma \\ 
&\leq \frac1\pi\|\widehat{\varphi}\|_{\infty} \Big(\delta - \frac{\rho(E,F)}{c_D t}\Big) \\ 
&\leq  \frac1\pi \|\widehat{\varphi}\|_{\infty} \delta e^{\delta {c_D}} e^{-\rho(E,F) / t},
\end{align*}
which completes the proof.
\end{proof}

The next two results show that $E^p_{D,\psi}(\cV)=E^p_{D,\eta}(\cV)$ for suitable $\psi$ in $\Psi(S^o_\theta)$ and $\eta$ in $\widetilde{\Psi}(\R)$. The results rely on some technical off-diagonal estimates that we postpone until Section~\ref{section:ODE}. The first result is an extension of \cite[Theorem~4.9]{AMcR}.

\begin{Prop}\label{Prop: MainEst}
Suppose that $M$ is a doubling metric measure space satisfying~\eqref{D}, that $D$ is a self-adjoint operator on $L^2(\cV)$, and the group $(e^{itD})_{t\in\R}$ has finite propagation speed. If $p\in[1,2]$, $\theta\in(0,\pi/2)$, $N\in\N$, $N>\kappa/2$, $\psi\in\Psi_{2N+1}^{N+1}(S_\theta^o)$, $\eta\in\widetilde{\Psi}_{N}(\R)$ and $\tilde\eta\in\widetilde{\Psi}(\R)$, then
\[
\|\cQ^D_{\tilde\eta}\cS^D_{\psi}U\|_{T^p} \lesssim \|U\|_{T^p}
\quad\text{and}\quad
\|\cQ^D_{\psi}\cS^D_{\eta}U\|_{T^p} \lesssim \|U\|_{T^p}
\]
for all $U\in T^p\cap T^2$.
\end{Prop}

\begin{proof}
The proof follows~\cite[Theorem~4.9]{AMcR}. When $p=2$, the result is immediate. When $p=1$, it suffices to show that there exists $C>0$ such that
\begin{equation}\label{eq: at.red.main}
\|\cQ^D_{\tilde\eta}\cS^D_\svt{\psi}(A)\|_{T^1} \leq C
\quad\text{and}\quad
\|\cQ^D_{\psi}\cS^D_{\eta}(A)\|_{T^1} \leq C
\end{equation}
for all $A$ that are $T^1$-atoms, since Theorem~\ref{maintentatomic} applies. When $p\in(1,2)$, the result then follows by the interpolation in~\eqref{TpInterp}. Therefore, it remains to prove~\eqref{eq: at.red.main}.

Lemma~\ref{Lem: ODestSchwartz+Psi} applied with $(m,n,N,\sigma,\tau,\delta)=(1,N,1,2N+1,N+1,1)$ shows that
\[
\|\ca_E(\tilde{\eta}_t\psi_s)(D )\ca_F\| \leq C
  \begin{cases}
    (s/t)^{N} \langle t/\rho(E,F)\rangle^{N},
    &\textrm{if } 0<s\leq t; \\
    (t/s) \langle s/\rho(E,F)\rangle^{2N+1},
    &\textrm{if } 0<t\leq s,
  \end{cases}
\]
for all measurable sets $E, F\subseteq M$. Since $(\psi_t\eta_s)(D)=(\eta_s\psi_t)(D)$, Lemma~\ref{Lem: ODestSchwartz+Psi} applied with $(m,n,N,\sigma,\tau,\delta)=(N,1,N,2N+1,N+1,1)$ also shows that 
\[
\|\ca_E(\psi_t\eta_s)(D )\ca_F\| \leq C
  \begin{cases}
    (s/t)^{N} \langle t/\rho(E,F)\rangle^{3{N}},
    &\textrm{if } 0<s\leq t; \\
    (t/s) \langle s/\rho(E,F)\rangle^{2N-1},
    &\textrm{if } 0<t\leq s,
  \end{cases}
\]
for all measurable sets $E, F\subseteq M$. These estimates combined with~\eqref{D} prove~\eqref{eq: at.red.main} as in Step~2 of the proof of Theorem~4.9 in~\cite{AMcR}.  This completes the proof.
\end{proof}

The second result is an extension of \cite[Lemma 5.2]{AMcR}.

\begin{Prop}\label{Prop: AMcRLem5.2FPS}
Suppose that $M$ is a doubling metric measure space satisfying~\eqref{D}, that $D$ is a self-adjoint operator on $L^2(\cV)$, and the group $(e^{itD})_{t\in\R}$ has finite propagation speed. If $p\in[1,2]$, $\theta\in(0,\pi/2)$, $\beta>\kappa/2$, $N\in\N$, $N>\kappa/2$, and all of $\psi \in \Psi_{\beta}(S^o_\theta)$, $\eta \in \widetilde{\Psi}_{N}(\R)$ and $\tilde\eta \in \widetilde{\Psi}(\R)$ are nondegenerate, then
\begin{equation}\label{eq:SFPS}
\cS^D_\psi(T^p\cap T^2) = \cS^D_\eta(T^p\cap T^2) = \{u\in\overline{\Ran(D)} : \cQ^D_{\tilde\eta} u \in T^p\}
\end{equation}
with the norm equivalence
\begin{equation}\label{eq:QFPS}
\|u\|_{E^p_{D,\psi}} \eqsim \|u\|_{E^p_{D,\eta}} \eqsim \|\cQ^D_{\tilde\eta} u\|_{T^p} \qquad \forall u\in E^p_{D,\psi}=\cS^D_{\psi}(T^p\cap T^{2}).
\end{equation}
Moreover, if the completion $H^p_{\cD,\psi}$ of $E^p_{\cD,\psi}$ in $L^p$ exists, and $H^p_{\cD,\psi}\cap L^2=E^p_{\cD,\psi}$, then there are unique extensions $\cS^D_{\eta}\in\mathcal{L}(T^p,H^p_{D,\psi})$ and  $\cQ^D_{\tilde\eta} \in \mathcal{L}(H^p_{D,\psi},T^p)$, and $H^p_{D,\psi} = \cS^D_{\eta}(T^p)$ with the norm equivalence
\begin{equation}\label{eq:HpEquivNormsFPS}
\|u\|_{H^p_{D,\psi}}
\eqsim 
\inf\{\|U\|_{T^p} : U\in T^p \text{ and } u=\cS^D_{\eta} U\}
\eqsim \|\cQ^D_{\tilde\eta}u\|_{T^p} \quad \forall u\in H^p_{D,\psi}.
\end{equation}
\end{Prop}

\begin{proof}
It suffices, by Theorem~\ref{Thm: AMcR_4.9_5.2}, to prove the result for a fixed nondegenerate $\psi$ in $\Psi_{\beta}(S^o_\theta)$, so we select $\psi$ in $\Psi_{2N+1}^{N+1}(S^o_\theta)$ satisfying $\int_0^\infty \psi(\pm t)^2 \frac{dt}{t}=1$. Suppose that both $\eta\in\widetilde{\Psi}_{N}(\R)$ and $\tilde\eta\in\widetilde{\Psi}(\R)$ are nondegenerate, and then use Corollary~\ref{Cor: Repr} to obtain $\varphi$ and $\tilde{\varphi}$ in $\Psi_{2N+1}^{N+1}(S^o_\theta)$ such that $\cS^D_\eta\cQ^D_\varphi = \cS^D_{\tilde \varphi}\cQ^D_{\tilde\eta} = \mathsf{P}_{\overline{\Ran(D)}} = \cS^D_{\psi}\cQ^D_{\psi}$. The proof of  \eqref{eq:SFPS} and \eqref{eq:QFPS} proceeds in three parts corresponding to the set inclusions
\[
\cS^D_\psi(T^p\cap T^2) \overbrace{\subseteq}^{\text{(i)}} \cS^D_\eta(T^p\cap T^2)
\overbrace{\subseteq}^{\text{(ii)}} \{u\in\overline{\Ran(D)} : \cQ^D_{\tilde\eta} u \in T^p\} \overbrace{\subseteq}^{\text{(iii)}} \cS^D_\psi(T^p\cap T^2)
\]
and the related norm estimates.

(i) If $u\in \cS^D_{\psi}(T^p\cap T^2)$, then~\eqref{eq:AMcR5.2Sets} implies that $u\in\overline{\Ran(D)}$ and $\cQ^D_{\psi} u \in T^p\cap T^2$, so $u=\cS^D_\eta(\cQ^D_\varphi \cS^D_{\psi} \cQ^D_{\psi} u)$ and~\eqref{eq:AMcR4.9} followed by \eqref{eq:AMcR5.2Norms} imply that
\[
\|u\|_{E^p_{D,\eta}} \leq \|\cQ^D_\varphi \cS^D_{\psi} (\cQ^D_{\psi} u)\|_{T^p} \lesssim \|\cQ^D_{\psi} u\|_{T^p} \eqsim \|u\|_{E^p_{D,{\psi}}}.
\]

(ii) If $u\in \cS^D_\eta(T^p\cap T^2)$, then $u\in\overline{\Ran(D)}$ by~\eqref{eq:BPfcRan}, and there exists $V \in T^p\cap T^2$ such that $u=\cS^D_\eta(V)$ and $\|V\|_{T^p}\leq 2\|u\|_{E^p_{D,\eta}}$, so by applying Proposition~\ref{Prop: MainEst} twice we obtain
\[
\|\cQ^D_{\tilde\eta}u\|_{T^p} = \|\cQ^D_{\tilde\eta} \cS^D_{\psi} (\cQ^D_{\psi} \cS^D_\eta V)\|_{T^p} \lesssim \|V\|_{T^p} \lesssim \|u\|_{E^p_{D,\eta}}.
\]

(iii) If $u\in\overline{\Ran(D)}$ and $\cQ^D_{\tilde\eta} u \in T^p$, then $u=\cS^D_{\psi} (\cQ^D_{\psi} \cS^D_{\tilde \varphi} \cQ^D_{\tilde\eta} u)$, so~\eqref{eq:AMcR4.9} implies that
\[
\|u\|_{E^p_{D,{\psi}}} \leq \|\cQ^D_{\psi} \cS^D_{\tilde \varphi} (\cQ^D_{\tilde\eta} u)\|_{T^p} \lesssim \|\cQ^D_{\tilde\eta} u\|_{T^p}.
\]
We obtain \eqref{eq:HpEquivNormsFPS} by the arguments used to prove \eqref{eq:HpEquivNorms}. This completes the proof.
\end{proof}

We now introduce hypothesis~\ref{SFinLqFPS} on $D$ in order to prove that the completion of $E^p_{D,\psi}(\cV)$ in $L^p(\cV)$ exists. This provides an alternative to hypothesis~\ref{SFinLq} from Theorem~\ref{Thm: H1injectivity} when $D$ is self-adjoint and $(e^{itD})_{t\in\R}$ has finite propagation speed.  The advantage of hypothesis~\ref{SFinLqFPS} is that $\cS_\eta^DF$ has compact support whenever $F$ has compact support, and as such, it is more easily verified that $\cS_\eta^DF \in L^{q'}(\cV)$.

\begin{Thm}\label{Thm: H1injectivityFPS}
Suppose that $M$ is a doubling metric measure space satisfying~\eqref{D}, that $D$ is a self-adjoint operator on $L^2(\cV)$, and $(e^{itD})_{t\in\R}$ has finite propagation speed. If $1\leq q\leq p\leq 2$, $\theta\in(0,\pi/2)$, $\beta>\kappa/2$, $\psi\in\Psi_{\beta}(S^o_\theta)$ is nondegenerate and
\begin{equation}\tag*{(H4)$_{\widetilde\Psi}$}\label{SFinLqFPS}
\begin{minipage}[c]{0.85\textwidth}
there exists a nondegenerate function $\eta\in\widetilde{\Psi}(\R)$ such that the set\\
$\{F \in T^2\cap T^{q'} : \cS^D_\eta F \in L^{q'}(\cV)\}$ is weak-star dense in $T^{p'}(\cV_+)$, 
\end{minipage}
\end{equation}
where $1/q+1/q'=1$, then the completion $H^p_{D,\psi}(\cV)$ of $E^p_{D,\psi}(\cV)$ in $L^p(\cV)$ exists. Moreover, it holds that $H^p_{D,\psi}(\cV)\cap L^2(\cV)=E^p_{D,\psi}(\cV)$.
\end{Thm}

\begin{proof}
Following the proof of Theorem~\ref{Thm: H1injectivity}, let $(u_n)_n$ denote a Cauchy sequence in $E^p_{D,\psi}$ that converges to 0 in $L^p$. We need to show that $(u_n)_n$ converges to 0 in $E^p_{D,\psi}$. To see this, fix $\eta$ in $\widetilde{\Psi}(\R)$ satisfying~\ref{SFinLqFPS}. For all $n\in\N$, we have by \eqref{eq:QFPS} that
\begin{equation}\label{eq:embeqFPS}
\|u_n\|_{E^p_{D,\psi}} \eqsim \|\cQ^D_\eta u_n\|_{T^p}.
\end{equation}
We conclude by repeating the proof of Theorem~\ref{Thm: H1injectivity} with~\eqref{eq:embeq} replaced by~\eqref{eq:embeqFPS} and $\cQ^{\cD}_{\tilde\psi}$ replaced by $\cQ^D_\eta$.
\end{proof}

\begin{Rem}\label{Rem: Cc}
Note that \ref{SFinLqFPS} holds whenever $\cS^{D}_{\eta}(C^\infty_c(\cV_+)) \subseteq L^{q'}(\cV)$, where $C^\infty_c(\cV_+)$ denotes the space of smooth compactly supported sections in $T^2(\cV_+)$. This is because $C^\infty_c(\cV_+)$ is weak-star dense in $T^{p'}(\cV_+)$ for all $p\in[1,2]$. To see this, a mollification argument can be applied in combination with Remark~\ref{Rem: Tc}.
\end{Rem}

\subsection{Atomic Theory}\label{section: Atoms}
We obtain a characterisation of $H^1_{D,\psi}(\cV)$ in terms of the atoms from Definition~\ref{moleculedef} and the space $H^1_{D,\text{at}(N)}(\cV)$ from Definition~\ref{AMcRDef6.1}.

\begin{Thm}\label{AMcRThm6.2FPS}
Suppose that $M$ is a doubling metric measure space satisfying~\eqref{D}, that $D$ is a self-adjoint operator on $L^2(\cV)$, and $(e^{itD})_{t\in\R}$ has finite propagation speed. Also, assume that for some $\theta\in(0,\pi/2)$, $\beta>\kappa/2$ and nondegenerate $\psi\in\Psi_{\beta}(S^o_\theta)$, the completion $H^1_{D,\psi}(\cV)$ of $E^1_{D,\psi}(\cV)$ in $L^1(\cV)$ exists, and $H^1_{D,\psi}(\cV)\cap L^2(\cV)=E^1_{D,\psi}(\cV)$. It follows that if $N\in\N$ and $N>\kappa/2$, then $H^1_{D,\psi}(\cV)=H^1_{D,\text{mol}(N)}(\cV)=H^1_{D,\text{at}(N)}(\cV)$.
\end{Thm}

\begin{proof}
Suppose that $N\in\N$ and $N>\kappa/2$. Theorem~\ref{AMcRThm6.2} and Lemma~\ref{Lem: ODestRes} show that $H^1_{D,\psi}=H^1_{D,\text{mol}(N)} \supseteq H^1_{D,\text{at}(N)}$. It remains to prove that $H^1_{D,\psi}\subseteq H^1_{D,\text{at}(N)}$. To do this, fix a nondegenerate $\eta$ in $\widetilde{\Psi}_N(\R)$. We claim that there exists $c>0$ such that $c\,\cS^D_\eta A$ is an $H^1_D$-atom of type $N$ whenever $A$ is a $T^1$-atom. The claim allows us to prove that $H^1_{D,\psi}\subseteq H^1_{D,\text{at}(N)}$ by repeating the proof of Theorem~\ref{AMcRThm6.2} with $\tilde\psi$ replaced by $\eta$ and then relying on~\eqref{eq:HpEquivNormsFPS} instead of \eqref{eq:HpEquivNorms}.

To prove the claim, let $A$ denote a $T^1$-atom and let $B$ denote a ball in $M$ with radius $r(B)>0$ such that $A$ is supported in the tent $T(B)$ and $\|A\|_{T^2} \leq \mu(B)^{-1/2}$. Note that $A_t$ is supported in $B$ when $t\in(0,r(B)]$, and that $\eta_t(D)A_t=0$ when $t>r(B)$. The finite propagation speed, in particular~\eqref{mn}, then implies that there exists $\alpha>0$, which only depends on $\eta$ and $D$, such that $\eta_t(D)A_t$ is supported in $\alpha B$ for all $t>0$, hence $\cS^D_\eta A$ is supported in~$\alpha B$.

Now set $\tilde\eta(x)=x^{-N}\eta(x)$ for all $x\in\R\setminus\{0\}$, and $\tilde\eta(0) = \partial^N\eta(0)/N!$, which equals $\lim_{x\rightarrow 0} x^{-N}\eta(x)$. Lemma~\ref{Lem: SchwartzMoments} shows that $\tilde\eta\in\widetilde{\Theta}(\R)$, and so the properties of the $B^\infty(\R)$ functional calculus imply that  the putative atom $a:=\cS^D_\eta A$ has the form
\[
a=\cS^D_\eta A = D^N\left(\int_0^\infty t^N \tilde\eta_t(D)A_t \frac{d t}{t} \right) =: D^N b.
\]
It remains to verify that $a$ and $b$ above satisfy the atomic bounds in Definition~\ref{moleculedef}. We use the doubling property to obtain
\[
\|a\|_2=\|\cS^D_\eta A\|_2 \lesssim \|A\|_{T^2} \leq \mu(B)^{-1/2} \lesssim \mu(\alpha B)^{-1/2},
\]
and since $A_t=0$ for all $t>r(B)$, we also have
\[
\|b\|_2 = \bigg\|\!\int_0^\infty t^N \tilde\eta_t(D)A_t \frac{d t}{t} \bigg\|_2
\!=\|\cS^D_{\tilde\eta}(t^NA_t)\|_2
\lesssim r(B)^N \|A\|_{T^2}
\lesssim (\alpha r(B))^N \mu(\alpha B)^{-1/2}.
\]
Therefore, there exists $c>0$, which does not depend on $A$, such that $c\, \cS^D_\eta(A)$ is an $H^1_D$-atom of type $N$. This proves the claim and completes the proof.
\end{proof}

\begin{Rem}
The proof of Theorem~\ref{AMcRThm6.2FPS} shows that the same result holds when the $L^1(\cV)$ convergence required in Definition~\ref{AMcRDef6.1} is replaced with $H^1_{D,\psi}(\cV)$ convergence.
\end{Rem}

\subsection{The Embedding $H^p_{D}\subseteq L^p$ for Smooth Differential Operators}\label{section: ApplicationsII}
We now consider the case when $M$ is a complete Riemannian manifold, which is assumed to be smooth (infinitely differentiable) and connected, with geodesic distance $\rho$ and Riemannian measure $\mu$. The vector bundle $\cV$ is also assumed to be \textit{smooth}, which means that the complex vector bundle $\pi:\cV\rightarrow M$ is equipped with a Hermitian metric $\langle\cdot,\cdot\rangle_x$ that is infinitely differentiable with respect to $x\in M$. Let $\dim(M)$ denote the dimension of $M$ and let $\dim(\cV)$ denote the fibre dimension of $\cV$. We prove a general result for a class of first-order differential operators on $L^2(\cV)$. The results for the Hodge--Dirac operator in Theorem~\ref{Thm: HpinLpIntro} are deduced afterwards.

A smooth-coefficient, first-order, differential operator $D_c$ is a linear operator on $L^2(\cV)$ with domain $\Dom(D_c) = C_c^\infty(\cV)$ such that on any coordinate patch over which $\cV$ is trivial, there are smooth, matrix-valued ($\mathcal{L}(\C^{\dim(\cV)})$-valued) functions $(A_j)_{j=0,\ldots,\dim(M)}$ such that the action of $D_c$ on that coordinate patch is given by the Euclidean operator $\sum_{j=1}^{\dim(M)} A_j \partial_j + A_0$. For each $x\in M$ in such a coordinate patch and  each ${\xi\in T^*_xM}$ given by $\xi=\sum_{j=1}^{\dim(M)}  \xi_j dx^j$, the \textit{principal symbol} $\sigma_{D_c}(x,\xi)$ is the endomorphism on the fibre $\cV_x$ given by $\sum_{j=1}^{\dim(M)} A_j\xi_j$.  A full account of these standard facts, including a coordinate-free definition of the principal symbol, is in~\cite[Chapter~IV, Section~2]{Wells2008}. Moreover, for any $\eta\in C_c^\infty(M)$, the principal symbol is given by the commutator $[D_c,\eta I]u = D_c(\eta u) - \eta D_cu$, since
\[
\big(\sigma_{D_c}(x,d\eta(x))\big)\big(u(x)\big) = \left([D_c,\eta I]u\right)(x) \quad \forall x\in M,\ \forall u\in C^\infty_c(\cV),
\]
where $d$ is the exterior derivative. 

An operator $D_c$ is called \textit{symmetric} when $\langle D_cu,v\rangle = \langle u,D_cv\rangle$  for all $u, v \in C_c^\infty(\cV)$. A symmetric first-order operator has a skew-symmetric principal symbol. Chernoff  proved in \cite{Chernoff1973} that if the principal symbol of a symmetric, smooth-coefficient, first-order, differential operator satisfies a certain bound, then the operator is essentially self-adjoint and  generates a group with finite propagation speed (related results are discussed in Remark~\ref{Rem: FPS}). This allows us to prove the following result. 

\begin{Thm}\label{Thm: HpinLpD}
Suppose that $M$ is a complete Riemannian manifold satisfying~\eqref{D}\! and that $\cV$ is a smooth vector bundle over $M$. Let $D$ denote the unique self-adjoint extension of a symmetric, smooth-coefficient, first-order, differential operator $D_c$ on $L^2(\cV)$ for which there exists $c_D>0$ such that the principal symbol $\sigma_{D_c}$ satisfies
\begin{equation}\label{eq: symbol}
\|\sigma_{D_c}(x,\xi)\|_{\mathcal{L}(\cV_x)} \leq c_D |\xi|_{T^*_xM} \qquad\forall x\in M,\ \forall \xi\in T^*_xM.
\end{equation}
If $p\in[1,2]$, $\theta\in(0,\pi/2)$, $\beta>\kappa/2$ and $\psi\in\Psi_{\beta}(S^o_\theta)$ is nondegenerate, then the completion $H^p_{D,\psi}(\cV)$ of $E^p_{D,\psi}(\cV)$ in $L^p(\cV)$ exists, and $H^p_{D,\psi}(\cV)\cap L^2(\cV)=E^p_{D,\psi}(\cV)$. Moreover, if $N\in\N$ and $N>\kappa/2$, then $H^1_{D,\psi}(\cV)=H^1_{D,\text{mol}(N)}(\cV)=H^1_{D,\text{at}(N)}(\cV)$.
\end{Thm}

\begin{proof}
The proof that $D_c$ is essentially self-adjoint on $L^2(\cV)$ is in~\cite[Theorem 2.2]{Chernoff1973}. The results of Chernoff~\cite[Theorem 1.3 and Corollary 1.4]{Chernoff1973} also show that the group $(e^{itD})_{t\in\R}$ has finite propagation speed $c_D$. Therefore, by Theorems~\ref{Thm: H1injectivityFPS} and~\ref{AMcRThm6.2FPS}, it suffices to prove that~\ref{SFinLqFPS} holds with $q=1$. 

First, we require a known estimate for the Sobolev spaces $W^{k,2}(\cV)$, where $k\in\N$. If $k>1+\dim(M)/2$ and $\cB$ is a ball in $M$, then there exists $C_{\cB}>0$ such that, for all $u\in W^{k,2}(\cV)$ with $\sppt(u)\subset \cB$, then 
\begin{equation}\label{sobolev}
\|u\|_{\infty} \leq C_{\cB} \|u\|_{W^{k,2}(\cV)}.
\end{equation}
This Sobolev embedding theorem can be found in \cite[Chapter~IV, Proposition~1.1]{Wells2008}.

Second, we require a known energy estimate. If $k\in\N$, $T>0$, and $\cB$ is a ball in $M$, then there exists $C_{T,\cB}>0$ such that, for all $u \in C^\infty_c(\cV)$ with $\sppt(u) \subset \cB$, then 
\begin{equation}\label{energyest}
\|e^{itD}u\|_{W^{k,2}(\cV)} \leq C_{T,\cB} \|u\|_{W^{k,2}(\cV)} \qquad \forall t\in[-T,T].
\end{equation}
This can be proved by the methods in \cite[Chapter IV, Section 2]{TaylorPsDO}, since $v(t)=e^{itD}u$ solves the initial value problem $\frac{dv}{dt}=iDv$ with $v(0)=u$.

Now choose a nondegenerate $\eta$ in $\widetilde{\Psi}(\R)$ and $\delta>0$ such that $\sppt \widehat{\eta}\subseteq [-\delta,\delta]$. Fix $k\in\N$ such that $k>1+\dim(M)/2$ and set $\alpha=1+c_D\delta$. For each $F \in C^\infty_c(\cV_+)$, there is a ball $B \subseteq M$ and $r>1$ such that $\sppt (F) \subseteq B \times [1/r,r]$. It follows that $\sppt (e^{istD} F_t) \subseteq (1+c_D|s|t/r)B \subseteq \alpha B$ for all $s\in[-\delta,\delta]$ and $t\in[1/r,r]$. Hence
\begin{align*}
\|\cS_\eta^D F\|_\infty &= \left\| \int_{1/r}^r \left(\frac{1}{2\pi} \int_{-\delta}^\delta \widehat{\eta}(s) e^{istD}F_t\, ds\right) \frac{dt}{t} \right\|_\infty \\
&\lesssim \int_{1/r}^r \int_{-\delta}^\delta \|e^{istD}F_t\|_\infty \, ds\, dt \\
&\lesssim \int_{1/r}^r \int_{-\delta}^\delta \|e^{istD}F_t\|_{W^{k,2}(\cV)} \, ds\, dt \\
&\lesssim \int_{1/r}^r \|F_t\|_{W^{k,2}(\cV)} \, dt \\
&<\infty,
\end{align*}  
where the first line uses~\eqref{FT}, the third line uses~\eqref{sobolev} with $\cB = \alpha B$, the fourth line uses~\eqref{energyest} with $\cB=B$, and the fifth line uses the continuity of $F$ in $C^\infty_c(\cV_+)$. This shows that $\cS_\eta^D (C^\infty_c(\cV_+)) \subseteq L^\infty(\cV)$, so Remark~\ref{Rem: Cc} implies that \ref{SFinLqFPS} holds with $q=1$. This completes the proof.
\end{proof}

\begin{Rem}\label{Rem: FPS}
McIntosh and Morris~\cite[Theorem~1.1]{McM} proved recently that any $C_0$-group $(e^{itD})_{t\in\R}$ generated by a first-order system $D$ satisfying~\eqref{eq: symbol} has finite propagation speed. In particular, finite propagation speed for such groups is not restricted to smooth-coefficient nor self-adjoint systems.
\end{Rem}

\begin{Rem}
The proof of Theorem~\ref{Thm: HpinLpD} extends immediately to an analogous class of first-order pseudodifferential operators but we shall not pursue this matter here.
\end{Rem}

We now prove Theorem~\ref{Thm: HpinLpIntro}, which fills a gap in the theory of Hardy spaces of differential forms developed by Auscher, McIntosh and Russ~\cite{AMcR}.

\begin{proof}[Proof of Theorem~\ref{Thm: HpinLpIntro}]
Let $M$ denote a doubling, complete Riemannian manifold. The bundle $\wedge T^*M=\oplus_{k=0}^{\dim(M)}\wedge^k T^*M$, where $\wedge^k T^*M$ denotes the $k\text{th}$ exterior power of the cotangent bundle $T^*M$, is defined with the Hermitian metric induced by the Riemannian metric. The Hodge--Dirac operator $D=d+d^*$ is defined initially on $C^\infty_c(\wedge T^*M)$, where $d$ and $d^*$ denote the exterior derivative and its adjoint. This is a symmetric, smooth-coefficient, first-order, differential operator on $L^2(\wedge T^*M)$ with principal symbol
\[
\sigma_D(x,\xi)\zeta = \xi\ext \zeta - \xi\lint \zeta \qquad \forall x\in M, \ \forall \xi\in T^*_xM, \ \forall \zeta\in\wedge T^*_xM,
\]
where $\ext$ and $\lint$ denote the exterior and (left) interior products on $\wedge T_x^*M$. These properties of the Hodge--Dirac operator are well known, and in particular, we have
\[
|\sigma_D(x,\xi)\zeta|_{\wedge T^*_xM}
= |\xi|_{T^*_xM}|\zeta|_{\wedge T^*_xM}\qquad \forall x\in M, \ \forall \xi\in T^*_xM, \ \forall \zeta\in\wedge T^*_xM,
\]
so the hypotheses of Theorem~\ref{Thm: HpinLpD} hold, and its conclusions imply Theorem~\ref{Thm: HpinLpIntro}.
\end{proof}

\section{The Embedding $H^p_L\subseteq L^p$ for Nonnegative Self-Adjoint Operators}\label{section: HLMMY}
We now combine the theory of the previous two sections to prove Theorem~\ref{Thm: HpinLPforHLMMY}. The atomic characterisation in Theorem~\ref{Thm: HpinLpIntroL} is then an immediate corollary. A new proof of Theorem~\ref{Thm: HpinLpIntroL} for smooth coefficient operators is also presented.

We return to the context of a vector bundle $\cV$ over a doubling metric measure space $M$. A nonnegative self-adjoint operator $L:\Dom(L)\subseteq L^2(\cV)\rightarrow L^2(\cV)$ is said to satisfy \textit{Davies--Gaffney estimates} when there exist constants $C,c>0$ such that
\begin{equation}\label{DG}
 \|\ca_E e^{-tL} \ca_F u \|_2 \leq C e^{-c\rho(E,F)^2/t} \|u\|_2
\end{equation}
for all $t>0$, all $u\in L^2(\cV)$ and all measurable sets $E,F\subseteq M$, where $(e^{-tL})_{t>0}$ is the analytic semigroup generated by $-L$. The following builds on the theory of Hardy spaces developed for such operators by Hofmann, Lu, Mitrea, Mitrea and Yan~\cite{HLMMY}.

\begin{proof}[Proof of Theorem~\ref{Thm: HpinLPforHLMMY}]
Since $L$ is self-adjoint, it satisfies \eqref{H1} and \eqref{H2} with $\omega=0$, $C_\theta=1/\sin\theta$ and $c_\theta=1$. We now prove that $L$ satisfies~\eqref{H3} with $m=2$. Let $E$ and $F$ denote measurable subsets of $M$. Since $L$ is nonnegative and self-adjoint, the Davies--Gaffney estimate \eqref{DG} is equivalent to the property that the cosine group $\cos(t\sqrt{L}) := \frac{1}{2}(e^{it\sqrt{L}} + e^{-it\sqrt{L}})$ has finite propagation speed (see~\cite[Theorem~2]{Sik} and \cite[Theorem~3.4]{CoulhonSikora}), where $(e^{it\sqrt{L}})_{t\in\R}$ is the $C_0$-group generated by the skew-adjoint operator $i\sqrt{L}$. Therefore, there exists $c_L>0$ such that $\ca_E \cos(t\sqrt{L}) \ca_F=0$ whenever $\rho(E,F)> c_L |t|$. For all $z\in\C$ with $\im (\pm z) > 0$, we use the integral representation $(zI-L)^{-1}= \frac{\pm 1}{i\sqrt{z}}\int_0^\infty e^{\pm i\sqrt{z}t}\cos(t\sqrt{L})\, dt$ (see \cite[Example~3.14.15]{ArendtBattyHieberNeubrander2011}) to obtain
\begin{align*}
\|\ca_E (zI-L)^{-1} \ca_F\| 
&\leq \frac{1}{|z|^{1/2}} \int_{\rho(E,F)/{c_L}}^\infty |e^{\pm i\sqrt{z}t}| \|\ca_E  \cos(t\sqrt{L}) \ca_F\|\, d t \\
&\leq \frac{1}{|z|^{1/2}}\int_{{\rho(E,F)}/{{c_L}}}^\infty e^{-(\im(\pm \sqrt{z}))t} \, dt.
\end{align*}
It is understood here that $\sqrt{z}= |z|^{1/2} e^{i\Arg(z)/2}$ with $\Arg(z)\in(-\pi,\pi]$, so then $\im(\sqrt{z}) = |z|^{1/2}\sin(\Arg(z)/2)$, and for each $\theta\in(0,\pi/2)$, it follows that
\[
\|\ca_E (zI-L)^{-1} \ca_F\| 
\leq \frac{C_{\theta/2}}{|z|} \exp\left(-\frac{\rho(E,F)|z|^{1/2}}{{c_L} C_{\theta/2}}\right) \qquad \forall z\in \C\setminus S_{\theta},
\]
which implies \eqref{H3} with $m=2$.

We have now shown that $L$ satisfies \eqref{H1}--\eqref{H3} with $\omega=0$ and $m=2$, and since $L$ satisfies \eqref{GnewSA}, hypothesis \ref{SFinLq} holds with $q=1$ by \eqref{H4est}. Therefore, except for the atomic characterisation, Theorems~\ref{Thm: H1injectivity} and~\ref{AMcRThm6.2} complete the proof.

It thus remains to prove that $H^1_{L,\psi}\subseteq H^1_{L,\text{at}(N)}$ when $\psi\in\Psi_\beta(S^o_\theta)$ and $N>\kappa/4$. Let $\tilde\psi(z)=ze^{-z}$ on $S^o_{\theta}$ and fix a nondegenerate \textit{even} function $\eta$ in $\widetilde{\Psi}_{2N}^\delta(\R)$ such that
\[
\int_0^\infty \eta(tz) \tilde\psi(t^2z^2) \, \frac{dt}{t} = 1 \qquad \forall z\in S^o_{\theta/2}.
\]
For example, choose any nondegenerate, even, real-valued function $\varphi\in C_c^\infty(\R)$ supported on $[-\delta/2,\delta/2]$ and let $\eta (x) = \alpha\,|x^{N}\widehat\varphi(x)|^2$ for all $x\in\R$, where $\alpha$ is the normalizing constant defined by $\alpha \int_0^\infty t^{2N}\widehat\varphi(t)^2 t^2 e^{-t^2}\frac{dt}{t} =1$.

Applying Proposition~\ref{Prop: Calderon} with $D=\sqrt{L}$, we obtain
\[
\cS^{\sqrt{L}}_{\eta}\cQ^{L}_{\tilde\psi}u := \int_0^\infty \eta(t\sqrt{L})\tilde\psi(t^2L)u \, \frac{dt}{t}
= u \qquad \forall u \in \overline{\Ran(L)}.
\]
The operator $\cQ^{L}_{\tilde\psi}$ has an extension $\cQ^{L}_{\tilde\psi} \in \mathcal{L}(H^1_{L,\psi},T^1)$ by \eqref{eq:HpEquivNorms}, since we have already established the embedding $H^1_{L,\psi}\subseteq L^1$ and that $H^1_{L,\psi}\cap L^2 = E^1_{D,\psi}$. It is also the case that $\cS^{\sqrt{L}}_{\eta}$ has an extension $\cS^{\sqrt{L}}_{\eta}\in\mathcal{L}(T^1,H^1_{L,\psi})$, but to prove this we must modify the theory in Section~\ref{section: FPS} to incorporate the finite propagation of the cosine group $\cos(t\sqrt{L})$. To this end, the fact that $\eta$ is an even function allows us to write 
\[
\eta_t(\sqrt{L})u= \frac{1}{\pi} \int_{0}^\infty \widehat{\eta_t}(s) \cos(s\sqrt{L})u\, d s \quad \forall t>0, \ \forall u \in L^2(\cV).
\]
We then follow the proof of Lemma~\ref{Lem: ODestSchwartz}, but instead use the finite propagation of the cosine group, to deduce that 
\begin{equation}\label{eq: Lmn}
\|\ca_E \eta_t(\sqrt{L}) \ca_F\| \leq \tfrac{1}{\pi} \|\widehat{\eta}\|_{\infty} \max\left\{\delta-\frac{\rho(E,F)}{c_Lt},0\right\} \qquad \forall t>0,\ \forall E, F\subseteq M.
\end{equation}
The extension $\cS^{\sqrt{L}}_{\eta}\in\mathcal{L}(T^1,H^1_{L,\psi})$ is then obtained as in Propositions~\ref{Prop: MainEst} and~\ref{Prop: AMcRLem5.2FPS}.

Now let $u\in H^1_{L,\psi}$. It follows from above that $u=\cS^{\sqrt{L}}_{\eta} U$, where $U:=\cQ^{L}_{\tilde\psi}u \in T^1$. Therefore, in order to show that $u\in H^1_{L,\text{at}(N)}$, it suffices to show that $\cS^{\sqrt{L}}_\eta A$ is an $H^1_L$-atom of type $N$ whenever $A$ is a $T^1$-atom (see the reasoning in the proof of the atomic characterisation in Theorem~\ref{AMcRThm6.2FPS}). To do this, note that when $A$ is supported in the tent $T(B)$ over a ball $B\subseteq M$, then \eqref{eq: Lmn} implies that $\eta(t\sqrt{L})A_t$ is supported in $\alpha B$ for all $t>0$, where $\alpha>0$ only depends on $\eta$ and $L$. Following the proof of Theorem~\ref{AMcRThm6.2FPS}, we write $a:=\cS^{\sqrt{L}}_\eta A = (\sqrt{L})^{2N}(\int_0^\infty t^{2N} \tilde\eta(t\sqrt{L})A_t \frac{d t}{t}) =: L^N b$ for a suitable $\tilde\eta\in\widetilde{\Theta}(\R)$, and then verify that $a$ and $b$ satisfy the atomic bounds in Definition~\ref{moleculedef}. This proves that $H^1_{L,\psi}\subseteq H^1_{L,\text{at}(N)}$, which completes the proof.
\end{proof}

We conclude by presenting a new proof of the results in Theorem~\ref{Thm: HpinLpIntroL} that does not rely explicitly on the ultracontractivity estimate \eqref{Gnew} but instead requires that $A$ is self-adjoint with smooth coeffecients.

\begin{proof}[Proof of Theorem~\ref{Thm: HpinLpIntroL} when $A$ is self-adjoint with smooth coeffecients]\label{Eg: dAg}
Let  $M=\R^n$ and consider $L = -\Div A \nabla$ on $L^2(\cV) = L^2(\R^n)$, where $A \in L^\infty(\R^n,\cL(\C^n))$ has $C^\infty(\R^n)$ coefficients and is elliptic in the sense that there exists $\lambda>0$ such that 
\[
\langle A(x)\zeta,\zeta\rangle_{\C^n}\geq\lambda|\zeta|^2 \quad \forall \zeta\in\C^n,\ \forall x\in\R^n.
\]
This ellipticity condition, which is stronger than~\eqref{ellip}, implies that the matrix $A(x)$ is strictly positive and Hermitian. We proceed by introducing a first-order system~$D$, a multiplication operator $B$, and a vector bundle $\cV_B$, such that $L$ is a component of~$(BD)^2$, and $BD$ satisfies the hypotheses of Theorem~\ref{Thm: HpinLpD} on $L^2(\cV_B)$.

Let $D_c:C_c^\infty(\R^n,\C^{1+n}) \!\rightarrow\! C_c^\infty(\R^n,\C^{1+n})$ denote the symmetric, smooth-coefficient, first-order, differential operator on $L^2(\R^n,\C^{1+n})$ defined by
\[
D_c=\left[\begin{array}{cc} 0&-\Div \\ \nabla &0 \end{array}\right]:
\begin{array}{c}C_c^\infty(\R^n)\\ \oplus\\ C_c^\infty(\R^n,\C^{n})\end{array}
\rightarrow
\begin{array}{c}C_c^\infty(\R^n)\\ \oplus\\ C_c^\infty(\R^n,\C^{n})\end{array},
\]
where $\nabla f = (\partial_1f,\ldots,\partial_nf)$ and $\Div(u_1,\ldots,u_n) = \sum_{j=1}^n\partial_ju_j$. The principal symbol
\[
\sigma_{D_c}(x,\xi) = \left[\begin{array}{cc}0 &-\xi^T \\  \xi &0 \end{array}\right]
\qquad \forall x\in\R^n,\ \forall \xi \in \C^{n}
\]
satisfies \eqref{eq: symbol}, so the unique self-adjoint extension of $D_c$ is the operator
\[
D=\left[\begin{array}{cc}
0&-\Div \\ \nabla &0 \end{array}\right]:\begin{array}{c}W^{1,2}(\R^n)\\ \oplus\\ \Dom(\Div)\end{array}\subseteq \begin{array}{c}L^2(\R^n)\\ \oplus\\ L^2(\R^n,\C^n)\end{array}\to \begin{array}{c}L^2(\R^n)\\ \oplus\\ L^2(\R^n,\C^n),\end{array}
\]
where $\nabla$ denotes the gradient extended to $W^{1,2}(\R^n)$ and $\Div:=-\nabla^*$.

Let $B(x)= \left[\begin{array}{cc} 1&0\\0& A(x) \end{array}\right]$, so $B \in L^\infty(\R^n,\cL(\C^{1+n})) \cap C^\infty(\R^n,\cL(\C^{1+n}))$ and
\begin{equation}\label{BDL}
BD=  \left[\begin{array}{cc}0&-\Div \\ A\nabla &0 \end{array}\right]
\qquad \text{and} \qquad 
(BD)^2= \left[\begin{array}{cc}L &0 \\  0 &\tilde{L}\end{array}\right],
\end{equation}
where $\tilde{L} := -A\nabla \Div$.

Let $\cV_B$ denote the trivial bundle over $\R^n$ that has $\C^{1+n}$-valued sections and the smooth Hermitian metric $\langle \xi,\zeta\rangle_{(\cV_B)_x} := \langle B(x)^{-1}\xi, \zeta\rangle_{\C^{1+n}}$ for $x\in\R^n$ and $\xi,\zeta \in \C^{1+n}$ (since $B(x)$ is strictly positive and Hermitian, $B(x)^{-1}$ and $B(x)^{-1/2}$ are Hermitian; also $B^{-1}, B^{-1/2} \in L^\infty(\R^n,\cL(\C^{1+n})) \cap C^\infty(\R^n,\cL(\C^{1+n}))$). For $p\in[1,2]$, the space $L^p(\cV_B)$ is then the set $L^p(\R^n,\C^{1+n})$ together with the norm
\[
\|u\|_{L^p(\cV_B)} := \left(\int_{\R^n} |B(x)^{-1/2}u(x)|_{\C^{1+n}}^p\, dx\right)^{1/p} \eqsim \|u\|_{L^p(\R^n,\C^{1+n})} \quad \forall u\in L^p(\R^n,\C^{1+n}). 
\]

We now verify the hypotheses of Theorem~\ref{Thm: HpinLpD} for the system $BD_c$ on $L^2(\cV_B)$. The inner product on $L^2(\cV_B)$ is given by $\langle B^{-1}u, v \rangle_{L^2(\R^n,\C^{1+n})}$, so $BD_c$ is symmetric on $L^2(\cV_B)$. The principal symbol satisfies $\sigma_{BD_c}(x,\xi) = B(x)\sigma_{D_c}(x,\xi)$ and
\[
|\sigma_{BD_c}(x,\xi)\zeta|_{(\cV_B)_x} = |B(x)^{1/2}\sigma_{D_c}(x,\xi)\zeta|_{\C^n} \leq \|B\|_\infty^{1/2} |\xi|_{\C^n}|\zeta|_{\C^{1+n}} \leq \|B\|_\infty |\xi|_{\C^n}|\zeta|_{(\cV_B)_x}
\]
for all $x\in\R^n$, $\xi \in \C^{n}$ and $\zeta \in \C^{1+n}$, so $BD_c$ satisfies \eqref{eq: symbol} on $\cV_B$, as required.

We can now apply Theorem~\ref{Thm: HpinLpD}. In particular, consider $p\in[1,2]$, $\theta\in(0,\pi/2)$ and $\beta>n/4$. Fix a nondegenerate $\psi\in\Psi_{\beta}(S^o_\theta)$, and let $\tilde{\psi}(z)=\psi(z^2)$ on $S^o_{\theta/2}$ (thus $\tilde\psi\in\Psi_{2\beta}(S^o_{\theta/2})$ and $2\beta>n/2$). The completion $H^p_{BD,\tilde\psi}(\cV_B)$ of $E^p_{BD,\tilde\psi}(\cV_B)$ in $L^p(\cV_B)$ then exists, and $H^p_{BD,\tilde\psi}(\cV_B)\cap L^2(\cV_B) = E^p_{BD,\tilde\psi}(\cV_B)$, by Theorem~\ref{Thm: HpinLpD}.

We now use the fact that $L$ is a component of $(BD)^2$ to complete the proof. Note that $L$ satisfies \eqref{H1}--\eqref{H3} with $m=2$ (see Section~\ref{section: ApplicationsI}), so $E^p_{L,\psi}(\R^n)$ is defined with $m=2$, whereas $E^p_{BD,\tilde\psi}(\cV_B)$ is defined with $m=1$ (see Lemma~\ref{Lem: ODestRes}). Let $\varphi(z)=ze^{-z}$ on $S^o_{\theta}$, and let $\tilde\varphi(z)=\varphi(z^2)$ on $S^o_{\theta/2}$. We use \eqref{BDL} to write
\[
\tilde{\varphi}(tBD) = t^2(BD)^2e^{-t^2(BD)^2} 
=\left[\begin{array}{cc} t^2Le^{-t^2L} &0 \\ 0 &t^2\tilde{L}e^{-t^2\tilde{L}} \end{array}\right]
=\left[\begin{array}{cc}\varphi(t^2L) &0 \\ 0 &\varphi(t^2\tilde{L}) \end{array}\right]
\]
and then apply \eqref{eq:AMcR5.2Norms} to obtain
\begin{align*}
\|u\|_{E^p_{L,\psi}(\R^n)} \eqsim \|\varphi(t^2L)u\|_{T^p(\R^{n+1}_+)} 
\eqsim \left\|\tilde\varphi(tBD) \left[\begin{array}{c} u \\ 0 \end{array}\right]\right\|_{T^p((\cV_B)_+)} 
\eqsim \left\|\left[\begin{array}{c} u \\ 0 \end{array}\right]\right\|_{E^p_{BD,\tilde\psi}(\cV_B)}
\end{align*}
for all $u\in E^p_{L,\psi}(\R^n)$. The equivalence $L^p(\cV_B)\eqsim L^p(\R^n)$ and the results above for $H^p_{BD,\tilde\psi}(\cV_B)$ then imply that the completion $H^p_{L,\psi}(\R^n)$ of $E^p_{L,\psi}(\R^n)$ in $L^p(\R^n)$ exists, and $H^p_{L,\psi}(\R^n)\cap L^2(\R^n) = E^p_{L,\psi}(\R^n)$. Theorem~\ref{AMcRThm6.2} then provides the molecular characterisation of $H^1_{L,\psi}(\R^n)$. Moreover, if $N\in\N$ and $N>n/4$, then $H^1_{BD,\tilde\psi}(\cV_B)=H^1_{BD,\text{at}(2N)}(\cV_B)$ by Theorem~\ref{AMcRThm6.2FPS}, which implies that $H^1_{L,\psi}(\R^n)=H^1_{L,\text{at}(N)}(\R^n)$, since when $(a,\tilde{a})=(BD)^{2N}(b,\tilde{b})$ in $L^2(\R^n)\oplus L^2(\R^n,\C^n)$ is an $H^1_{BD}(\cV_B)$-atom of type $2N$, then $a=L^N b$ is an $H^1_{L}(\R^n)$-atom of type $N$ by \eqref{BDL}. This completes the proof.
\end{proof}

\section{Appendix: Off-Diagonal Estimates}\label{section:ODE}
This section contains technical estimates used to prove Propositions~\ref{Prop: MainEst} and~\ref{Prop: AMcRLem5.2FPS}. We begin with the following lemma, which allows us to manipulate $\widetilde\Psi(\R)$ class functions in a manner analogous to $\Psi(S^o_\theta)$ class functions.

\begin{Lem}\label{Lem: SchwartzMoments}
Suppose that $N\in\N$. The following hold.
\begin{enumerate}
\item For $n\in\N$ and $\varphi\in\widetilde{\Theta}(\R)$, the function ${\tilde\varphi}(x)\!:=x^n \varphi(x) \text{ for all } x \in\R$, is in $\widetilde{\Psi}_{n}(\R)$.\\
Moreover, if $\varphi\in\widetilde{\Psi}_N(\R)$, then ${\tilde\varphi}\in\widetilde{\Psi}_{N+n}(\R)$.
\item For $m\in\{1,\ldots,N\}$ and $\eta\in\widetilde{\Psi}_N(\R)$, the function
${\tilde\eta}(x):=x^{-m} \eta(x)$ for all $x\in\R\setminus\{0\}$, with ${\tilde\eta}(0):=\lim_{x\rightarrow 0} x^{-m} \eta(x) = \partial^m\eta(0)/m!$, is in $\widetilde{\Theta}(\R)$.\\
Moreover, if $m\in\{1,\ldots,N-1\}$, then ${\tilde\eta}\in\widetilde{\Psi}_{N-m}(\R)$ (and so ${\tilde\eta}(0)=0$).
\end{enumerate}
\end{Lem}

\begin{proof}
Suppose that $n\in\N$ and $\varphi\in\widetilde{\Theta}(\R)$. The function~${\tilde\varphi}$ defined in~(1) belongs to $\cS(\R)$ because $\varphi\in\cS(\R)$. The Fourier transform $\widehat{{\tilde\varphi}}$ is compactly supported because $\widehat{\varphi}$ is compactly supported and $\widehat{{\tilde\varphi}} = \partial^n\widehat{\varphi}$. For each $k\in\N$, there exist constants $c_{k,0},c_{k,1},\ldots, c_{k,k}$ such that
\[
\partial^k {\tilde\varphi}(x) = \sum_{j=0}^{\min\{n-1,k\}} c_{k,j} x^{n-j}\partial^{k-j}\varphi(x) + \sum_{j=n}^{k} c_{k,j}\partial^{k-j}\varphi(x), \qquad \forall x\in\R.
\]
It follows that $\partial^k {\tilde\varphi}(0)=0$ for all $k\in\{0,\ldots,n-1\}$, hence ${\tilde\varphi}\in\widetilde{\Psi}_{n}(\R)$. Moreover, if $\varphi\in\widetilde{\Psi}_N(\R)$, then $\partial^k {\tilde\varphi}(0)=0$ for all $k\in\{0,\ldots,N+n-1\}$, hence ${\tilde\varphi}\in\widetilde{\Psi}_{N+n}(\R)$. This proves~(1).

Now suppose that $m\in\{1,\ldots,N\}$ and $\eta\in\widetilde{\Psi}_N(\R)$. The function ${\tilde\eta}$ defined in~(2) satisfies the requirements of a Schwartz function, except possibly in a neighbourhood of the origin, because $\eta\in\cS(\R)$. The Paley--Wiener Theorem guarantees that $\eta$ has a  holomorphic extension to the entire complex plane, since $\widehat{\eta}$ is compactly supported. Therefore, there exists $\epsilon>0$ and a sequence $(a_j)_{j\in\N_0}$ such that the power series $a_0 + \sum_{j=1}^\infty a_jx^j$ converges to $\eta(x)$ for all $x\in[-\epsilon,\epsilon]$. The assumption that $\eta\in\widetilde{\Psi}_N(\R)$ implies that $a_j=0$ for all $j\in\{0,\ldots,N-1\}$, hence $a_{N}x^{N-m}+\sum_{j=N+1}^\infty a_jx^{j-m}$ converges to ${\tilde\eta}(x)$ for all $x\in[-\epsilon,\epsilon]$, and ${\tilde\eta}\in\cS(\R)$. Moreover, if $m\in\{1,\ldots,N-1\}$, then this also shows that $\partial^k {\tilde\eta}(0)=0$ for all $k\in\{0,\dots,N-m-1\}$. This proves~(2) provided that $\widehat{{\tilde\eta}}$ is compactly supported.

To show that $\widehat{{\tilde\eta}}$ is compactly supported when $m\in\{1,\ldots,N\}$, choose $\delta>0$ such that $\widehat{\eta}$ is supported in $[-\delta,\delta]$. It is enough to show that for each ${k}\in\{1,\dots,m\}$, there exist constants $c_{{k},0},c_{{k},1},\ldots,c_{{k},{k}-1}$ such that
\begin{equation}\label{eq: ind.ass}
\partial^{m-{k}}\widehat{{\tilde\eta}}(y) =
\begin{cases}
\displaystyle\sum_{j=0}^{{k}-1}c_{{k},j} y^{{k}-1-j}\int_{-\delta}^y x^j\widehat{\eta}(x)\, d x, & \quad\text{if}\quad |y|\leq\delta;\\
0,& \quad\text{if}\quad |y|>\delta,
\end{cases}
\end{equation}
since this proves that $\widehat{{\tilde\eta}}$ is compactly supported in $[-\delta,\delta]$ by setting ${k}=m$.

We prove~\eqref{eq: ind.ass} by induction. For $k=1$, since $\eta(x)=x^m{\tilde\eta}(x)$, we have $\widehat{\eta}=\partial^m\widehat{{\tilde\eta}}$, and so $\partial^{m-1}\widehat{{\tilde\eta}}(y) = \int_{-\infty}^y \widehat{\eta}(x)\, dx$. This shows that~\eqref{eq: ind.ass} holds for $k=1$, since $\widehat{\eta}$ is supported in $[-\delta,\delta]$ and $\int_{-\infty}^{\infty} \widehat{\eta}(x)\, dx=\eta(0)=0$. Next, assume that~\eqref{eq: ind.ass} holds for some ${k}={l}\in\{1,\dots,m-1\}$. Note that $\partial^{m-({l}+1)}\widehat{{\tilde\eta}}(y) = \int_{-\infty}^y \partial^{m-{l}}\widehat{{\tilde\eta}}(x)\, dx$. When $y<-\delta$, then $\partial^{m-({l}+1)}\widehat{{\tilde\eta}}(y)= 0$ by~\eqref{eq: ind.ass}. When $y\geq-\delta$, then we use~\eqref{eq: ind.ass} to obtain
\begin{align*}
\partial^{m-({l}+1)}\widehat{{\tilde\eta}}(y) &= \int_{-\delta}^{\min\{y,\delta\}} \Bigg(\sum_{j=0}^{{l}-1}c_{{l},j} x^{{l}-1-j}\int_{-\delta}^x w^j\widehat{\eta}(w)\, d w\Bigg) d x \\
&= \sum_{j=0}^{{l}-1}c_{{l},j} \int_{-\delta}^y \bigg(\int_{w}^{\min\{y,\delta\}} x^{{l}-1-j}\, d x \bigg) w^j\widehat{\eta}(w)\, d w \\
&= \sum_{j=0}^{{l}-1} \frac{c_{{l},j}}{{l}-j} \bigg( \min\{y,\delta\}^{{l}-j} \int_{-\delta}^y w^j\widehat{\eta}(w)\, d w 
- \int_{-\delta}^y w ^{{l}}\widehat{\eta}(w)\, d w\bigg).
\end{align*}
This shows that~\eqref{eq: ind.ass} holds for ${k}={l}+1$, since $\widehat{\eta}$ is supported in $[-\delta,\delta]$ and $\int_{-\infty}^\infty w^j\widehat{\eta}(w)\, d w = \partial^j\eta(0) = 0$ for all $j\in\{0,\ldots,N-1\}$. We then conclude that~\eqref{eq: ind.ass} holds for each ${k}\in\{1,\dots,m\}$. This completes the proof.
\end{proof}

We use a proof of Auscher and Martell~\cite[Theorem~2.3(b)]{AuscherMartell2007} to show that polynomial off-diagonal estimates are stable under composition. This allows us to combine the off-diagonal estimates for the $\Psi(S^o_\theta)$ class in~\eqref{eq:PfcOD} with those for the $\widetilde\Psi(\R)$ class in~\eqref{mn}. We use the notation $\langle \alpha \rangle = \min \{\alpha, 1\}$ and $\langle \frac{\alpha}{0} \rangle = 1$ when $\alpha>0$.

\begin{Lem}\label{Lem: ODestComp}
Suppose that $C,$ $\alpha >0$. If $\{T_t\}_{t>0}$ and $\{S_t\}_{t>0}$ are collections of operators in $\mathcal{L}(L^2(\cV))$ such that
\[
\|\ca_E T_t \ca_F\| \leq C\langle t/\rho(E,F)\rangle^{\alpha}
\quad\text{and}\quad
\|\ca_E S_t \ca_F\| \leq C\langle t/\rho(E,F)\rangle^{\alpha}
\]
for all $t>0$ and all measurable sets $E, F\subseteq M$, then there exists $\widetilde{C}>0$ such that
\[
\|\ca_E T_tS_s \ca_F\| \leq \widetilde{C} \langle \max\{s,t\}/\rho(E,F)\rangle^{\alpha}
\]
for all $s, t > 0$ and all measurable sets $E, F\subseteq M$.
\end{Lem}

\begin{proof}
Let $E, F\subseteq M$ denote measurable sets. The measure on $M$ is Borel with respect to the metric topology, so the set $\widetilde{E} = \{x\in M: \rho(x,E)\leq \rho(E,F)/2\}$ is closed and hence measurable. The result follows by writing
\[
\|\ca_E T_tS_s \ca_F\|
= \|\ca_E T_t(\ca_{\widetilde{E}}+\ca_{M\setminus {\widetilde{E}}})S_s \ca_F\|
\leq \|T_t\|\|\ca_{\widetilde{E}}S_s \ca_F\| + \|\ca_E T_t\ca_{M\setminus {\widetilde{E}}}\| \|S_s\|
\]
for all $s, t >0$, since $\rho(\widetilde{E},F)\geq \rho(E,F)/2$ and $\rho(E,M\setminus {\widetilde{E}}) \geq \rho(E,F)/2$.
\end{proof}

The following off-diagonal estimates are used to prove Propositions~\ref{Prop: MainEst} and~\ref{Prop: AMcRLem5.2FPS}.

\begin{Lem}\label{Lem: ODestSchwartz+Psi}
Suppose that $D$ is a self-adjoint operator on $L^2(\cV)$ and the group $(e^{itD})_{t\in\R}$ has finite propagation speed. If $m,n,N\in\N$ and $\delta,\sigma,\tau>0$ satisfy
\[
m\leq N,\quad m<\tau,\quad n<\sigma\quad\text{and}\quad \delta\in(0,\sigma-n),
\]
then for each $\eta\in\widetilde\Psi_N(\R)$ and $\psi\in\Psi_{\sigma}^\tau(S_\theta^o)$, there exists $C>0$ such that
\begin{equation}\label{odG}
\|\ca_E(\eta_t\psi_s)(D )\ca_F\| \leq C
  \begin{cases}
    (s/t)^{n} \langle t/\rho(E,F)\rangle^{\sigma-{n}-\delta},
    &\textrm{if } 0<s\leq t; \\
    (t/s)^{m} \langle s/\rho(E,F)\rangle^{\sigma+{m}-\delta},
    &\textrm{if } 0<t\leq s,
  \end{cases}
\end{equation}
for all measurable sets $E, F \subseteq M$.
\end{Lem}

\begin{proof}
Let $E, F\subseteq M$ denote measurable sets. Suppose that $0<s\leq t$ and define
\[
{\tilde\eta}(x)=x^{n}\eta(x)\ \forall x\in \R,\quad {\tilde\psi}(z)=z^{-{n}}\psi(z)\ \forall z \in S^o_\theta \quad\text{and}\quad {\tilde\psi}(0)=0.
\]
The function ${\tilde\eta}$ is in $ \widetilde\Psi_{N+{n}}(\R)$ by Lemma~\ref{Lem: SchwartzMoments}, so Lemma~\ref{Lem: ODestSchwartz} implies that
\begin{equation}\label{odF}
\|\ca_E{\tilde\eta}_t(D)\ca_F\| \lesssim e^{-\rho(E,F)/t} \lesssim \langle t/\rho(E,F)\rangle^{\sigma-{n}-\delta}.\end{equation}
The function ${\tilde\psi}$ is in $\Psi_{\sigma-{n}}^{\tau+{n}}(S^o_\theta)$, so~\eqref{eq:PfcOD} implies that
\begin{equation}\label{odP}
\|\ca_E {\tilde\psi}_s(D ) \ca_F\| \lesssim \langle s/\rho(E,F) \rangle^{\sigma-{n}-\delta}.
\end{equation}
We combine~\eqref{odF} and~\eqref{odP} using Lemma~\ref{Lem: ODestComp} to obtain
\begin{equation}\label{odC}
\|\ca_E {\tilde\eta}_t(D){\tilde\psi}_s(D ) \ca_F\|
\lesssim \langle t/\rho(E,F) \rangle^{\sigma-{n}-\delta}
\end{equation}
when $0<s\leq t$. The $B^\infty(\R)$ functional calculus is an algebra homomorphism and $\eta_t\psi_s = (s/t)^{n} {\tilde\eta}_t{\tilde\psi}_s$ on $\R$, where both ${\tilde\eta}_t$ and ${\tilde\psi}_s$ are in $B^\infty(\R)$. Therefore, we have $(\eta_t\psi_s)(D) = (s/t)^{n} {\tilde\eta}_t(D){\tilde\psi}_s(D)$,  and so~\eqref{odC} implies~\eqref{odG} when $0<s\leq t$. 

Now suppose that $0<t\leq s$ and define
\[
{\tilde{\tilde\eta}}(x)=x^{-{m}}\eta(x)\ \forall x \in \R\setminus\{0\},\ {\tilde{\tilde\eta}}(0)=\lim_{x\rightarrow 0} x^{-m} \eta(x) \ \text{and}\ {\tilde{\tilde\psi}}(z)=z^{m}\psi(z) \ \forall z\in S^o_\theta \cup \{0\}.
\]
The function ${\tilde{\tilde\eta}}$ is in $\widetilde\Theta(\R)$ by Lemma~\ref{Lem: SchwartzMoments}, since $m\leq N$ (note that ${\tilde{\tilde\eta}}(0)=\partial^m\eta(0)/m!$ and so we may have $\tilde{\tilde\eta}(0) \neq 0$ when $m=N$). Lemma~\ref{Lem: ODestSchwartz} then implies that $\|\ca_E{\tilde{\tilde\eta}}_t(D)\ca_F\|\lesssim e^{-\rho(E,F)/t}$. The function ${\tilde{\tilde\psi}}$ is in $\Psi_{\sigma+{m}}^{\tau-{m}}(S^o_\theta)$, so~\eqref{eq:PfcOD} implies that $\|\ca_E {\tilde{\tilde\psi}}_s(D) \ca_F\| \lesssim \langle s/\rho(E,F)\rangle^{\sigma+{m}-\delta}$. We also have $\eta_t\psi_s = (t/s)^{m} {\tilde{\tilde\eta}}_t{\tilde{\tilde\psi}}_s$ on $\R$, so by writing $(\eta_t\psi_s)(D)=(t/s)^{m} {\tilde{\tilde\eta}}_t(D){\tilde{\tilde\psi}}_s(D)$ and using Lemma~\ref{Lem: ODestComp} to combine the two preceding estimates, we obtain \eqref{odG} when ${0<t\leq s}$.
\end{proof}

\section*{Acknowledgements}
Auscher and Morris thank the Mathematical Sciences Institute at the Australian National University for support during the project. Auscher was also partially supported by the ANR project ``Harmonic Analysis at its Boundaries'', ANR-12-BS01-0013-01. McIntosh was supported by the Australian Research Council. We thank Lashi Bandara, Charles Batty, Andrea Carbonaro, Steve Hofmann and Pierre Portal for helpful conversations that improved the paper.

\end{document}